\documentclass{amsart}

\usepackage{amscd}
\usepackage{mathtools}
\usepackage{stmaryrd}

\theoremstyle{definition}
\newtheorem{theorem}{Theorem}[section]

\newtheorem{corollary}[theorem]{Corollary}
\newtheorem{proposition}[theorem]{Proposition}
\newtheorem{definition}[theorem]{Definition}

\newtheorem{remark}{Remark}
\newtheorem{claim}{Claim}

\usepackage{amsmath}
\usepackage{amsthm}
\usepackage{commath}
\usepackage{amssymb}
\usepackage{hyperref}
\usepackage{bbm}
\usepackage{tikz-cd}

\title{Non-strict plurisubharmonicity of energy on Teichm\"uller space}
\author{Ognjen To\v{s}i\'c}
\newcommand{\Teich}{\ensuremath{\mathcal{T}}}
\begin{document}

\begin{abstract}
    For an irreducible representation $\rho:\pi_1(\Sigma_g)\to\mathrm{GL}(n,\mathbb{C})$ there is an energy functional $\mathrm{E}_\rho:\Teich_g\to\mathbb{R}$, defined on Teichm\"uller space by taking the energy of the associated equivariant harmonic map into the symmetric space $\mathrm{GL}(n,\mathbb{C})/\mathrm{U}(n)$. It follows from a result of Toledo that $\mathrm{E}_\rho$ is plurisubharmonic, i.e. its Levi form is positive semi-definite. We study the kernel of this Levi form, and relate it to the $\mathbb{C}^*$ action on the moduli space of Higgs bundles. We also show that the points in $\Teich_g$ where strict plurisubharmonicity fails (i.e. this kernel is non-zero) are critical points of the Hitchin fibration. When $n\geq 2$ and $g\geq 3$, we show that for a generic choice $(S,\rho)$, the map $\mathrm{E}_\rho$ is strictly plurisubharmonic. We also describe the kernel of the Levi form for $n=1$.
\end{abstract}

    \maketitle

    \section{Introduction}
    Let $\Sigma_g$ be the closed surface of genus $g\geq 2$, and let $\rho:\pi_1(\Sigma_g)\to\mathrm{GL}(n,\mathbb{C})$ be a completely reducible representation. Let $X_n=\mathrm{GL}(n,\mathbb{C})/\mathrm{U}(n)$ be the symmetric space that corresponds to $\mathrm{GL}(n,\mathbb{C})$. For any Riemann surface $S\in\Teich_g$, by theorems of Corlette \cite{corlette} and Donaldson \cite{donaldson}, there exists a $\rho$-equivariant harmonic map $f:S\to X_n$. Using these harmonic maps, we define the energy functional $\mathrm{E}_\rho:\Teich_g\to\mathbb{R}$ that assigns to $S\in\Teich_g$ the energy of the map $f$. 
    \par This energy functional plays an important role in higher Teichm\"uller theory. One example is the Labourie conjecture, stating that when $\rho$ lies in a Hitchin component of a split real semisimple Lie group $G$, there exists a unique minimal $\rho$-equivariant disk in the symmetric space $X$ associated to $G$. Such minimal disks are in fact critical points of $\mathrm{E}_\rho$, and Labourie showed that $\mathrm{E}_\rho$ is proper \cite{Labourie2008}, and that the minimum is the unique critical point in rank $2$ \cite{labourie:hal-01141620}. The energy $\mathrm{E}_\rho$ also plays a role in the recent negative answer to the Labourie conjecture, due to Markovi\'c \cite{Markovic2022} in $\mathrm{PSL}(2,\mathbb{R})^3$, and Sagman--Smillie \cite{Sagman2022UnstableMS} in general, where non-uniqueness follows from the existence of saddle points of $\mathrm{E}_\rho$. 
    \par It was shown by Toledo \cite{Toledo2011HermitianCA} that the map $\mathrm{E}_\rho$ is plurisubharmonic. In this paper, we investigate directions where $\mathrm{E}_\rho$ is not strictly plurisubharmonic, i.e. the kernel of the Levi form of $\mathrm{E}_\rho$, that we denote $K_\rho$. Note that $K_\rho$ when $\rho:\pi_1(\Sigma_g)\to\mathbb{R}$ is a real cohomology class appears in the study of virtual properties of mapping class groups, specifically the Putman--Wieland conjecture, in the work of Markovi\'c and the author \cite{mt-hodge}.
    \par Our first result relates $K_\rho$ to the $\mathbb{C}^*$ action on the moduli space of Higgs bundles, another classical object.  Scaling the Higgs field by $i$ defines an order $4$ automorphism of the space of degree $0$ polystable Higgs bundles. By the non-abelian Hodge correspondence, this automorphism can be transported to the character variety $\mathrm{Rep}(\pi_1(\Sigma_g), \mathrm{GL}(n,\mathbb{C}))$, which is the GIT quotient of the space of representations $\pi_1(\Sigma_g)\to\mathrm{GL}(n,\mathbb{C})$ by the conjugacy action of $\mathrm{GL}(n,\mathbb{C})$. This automorphism depends on the choice of Riemann surface underlying the non-abelian Hodge theorem, so defines a map 
    \begin{align*}
        \mathcal{R}:\mathrm{Rep}(\pi_1(\Sigma_g),\mathrm{GL}(n,\mathbb{C}))\times\Teich_g\to \mathrm{Rep}(\pi_1(\Sigma_g),\mathrm{GL}(n,\mathbb{C})).
    \end{align*}
    We call $\mathcal{R}_\rho(S)=\mathcal{R}(\rho, S)$ the harmonic conjugate of $\rho$ on $S$.
    \begin{theorem}\label{thm:main-higgs}
        Let $\rho:\pi_1(\Sigma_g)\to\mathrm{GL}(n,\mathbb{C})$ be an irreducible representation. Then the space $K_\rho$ of directions in $T\Teich_g$ in which $\mathrm{E}_\rho$ is not strictly plurisubharmonic is exactly the space of directions annihilated by the derivative of $\mathcal{R}_\rho$.
    \end{theorem} 
    When $\rho$ is a Hitchin representation in $\mathrm{SL}(n,\mathbb{R})$, Slegers \cite{Slegers-strict} has shown that $\mathrm{E}_\rho$ is strictly plurisubharmonic. Our next result uses this to show that for a generic representation $\rho$, at a generic point in $\Teich_g$, the energy $\mathrm{E}_\rho$ is strictly plurisubharmonic. When $S$ is a marked Riemann surface and $\rho:\pi_1(\Sigma_g)\to\mathrm{GL}(n,\mathbb{C})$ is a completely reducible representation, denote by $\mathrm{Higgs}(\rho, S)$ the polystable degree $0$ Higgs bundle over $S$ that corresponds to $\rho$ by the non-abelian Hodge theorem.
    \begin{theorem}\label{thm:generic-special}
        Let $\rho:\pi_1(\Sigma_g)\to\mathrm{GL}(n,\mathbb{C})$ be an irreducible representation for $g\geq 3$. For any $S\in\Teich_g$, with the property that $\mathrm{Higgs}(\rho, S)$ lies in a smooth fibre of the Hitchin fibration, $K_\rho(S)=\{0\}$. Conversely, for any $g\geq 4, S\in\Teich_g$, and any $n\geq 2$, there exists a representation $\rho:\pi_1(\Sigma_g)\to\mathrm{GL}(n,\mathbb{C})$ such that $K_\rho(S)\neq \{0\}$.
    \end{theorem}
    \begin{remark} Here by a smooth fibre of the Hitchin fibration, we mean a fibre whose corresponding spectral curve is smooth. We refer the reader to \S\ref{sect:prelim} for the precise definitions.
    \end{remark} 
    \par When $n=1$, we are able to completely classify the kernel of the Levi form of $\mathrm{E}_\rho$ for any $\rho:\pi_1(\Sigma_g)\to\mathbb{C}^*=\mathrm{GL}(1,\mathbb{C})$. Before we state this result, recall that over a Riemann surface $S$, the space of holomorphic quadratic differentials $\mathrm{QD}(S)$ has dimension $3g-3$, and is naturally isomorphic to the cotangent space to $\Teich_g$ at $S$. We denote by $\Omega(S)$ the set of holomorphic 1-forms on $S$.
    \begin{proposition}\label{prop:n=1}
        Let $\rho:\pi_1(\Sigma_g)\to\mathbb{C}^*$ be a representation. Given a marked Riemann surface $S\in\Teich_g$, let $\phi$ be the holomorphic 1-form whose real part represents the cohomology class $-\frac{1}{2}\log\abs{\rho}$. Then $K_\rho(S)$ is the annihilator of the set $\phi\otimes\Omega(S)\leq\mathrm{QD}(S)$. Moreover, the distribution $K_\rho$ is integrable, and the leaves of the resulting foliation are complex submanifolds of $\Teich_g$ of codimension $g$.
    \end{proposition}
    Finally, we relate $K_\rho$ to the critical points of the Hitchin integrable system. Recall that, given a Riemann surface $S$, the moduli space of stable rank $n$ degree $0$ Higgs bundles over $S$, denoted $\mathcal{M}(S)$, is a symplectic manifold equipped with a (singular) fibration $H:\mathcal{M}(S)\to\mathcal{B}(S)$, introduced by Hitchin \cite{hitchin-int-systems}, such that all non-singular fibres of $H$ are Lagrangian submanifolds. We will recall a precise definition in \S\ref{subsubsec:hitchin-fibration}. 
    \begin{definition} 
        The $d$-th critical locus of $H$ is the set of points in $\mathcal{M}(S)$ where the rank of $D H$ is at most $\dim\mathcal{B}(S)-d$. 
    \end{definition}
     These appear in the recent work of Hitchin \cite{hitchin-rank-2}. He shows that in $\mathrm{SL}(2,\mathbb{C})$, the $d$-th critical locus is the space of Higgs fields that have at least $d$ zeros (counting multiplicity). 
    \begin{theorem}\label{thm:hitchin-integrable-system}
        Let $\rho:\pi_1(\Sigma_g)\to\mathrm{GL}(n,\mathbb{C})$ be an irreducible representation. Then if $d=\dim K_\rho(S)$, the point $\mathrm{Higgs}(\rho, S)$ lies in the $d$-th critical locus of the Hitchin integrable system.
    \end{theorem}
    \par All our results are derived from the following more precise analytic criterion. 
    \begin{theorem}\label{thm:higgs-lap-zero}
        Let $\rho:\pi_1(\Sigma_g)\to\mathrm{GL}(n,\mathbb{C})$ be completely reducible, $S$ be a marked Riemann surface of genus $g$. If $(E,\phi)=\mathrm{Higgs}(\rho, S)$, then $K_\rho(S)$ is the space of Beltrami forms $\mu$ on $S$ such that there exists a section $\xi$ of $\mathrm{End}(E)$ with
        \begin{align}\label{eq:higgs-lap-zero}
            \mu\phi=\bar{\partial}\xi\text{ and }[\phi,\xi]=0.
        \end{align}
    \end{theorem}
    We derive Theorem \ref{thm:higgs-lap-zero} in turn from a generalization of the analysis of Toledo \cite{Toledo2011HermitianCA}, described in \S\ref{subsec:results-on-riemannian}.
    \subsection{Results on Riemannian manifolds with very strongly seminegative curvature}\label{subsec:results-on-riemannian}
    Let $M$ be a Riemannian manifold with an isometry group $\mathrm{Isom}(M)$. Let $\rho:\pi_1(\Sigma_g)\to\mathrm{Isom}(M)$ be a representation.
    \begin{definition}
        If $(S_t\in\Teich_g:t\in\mathbb{D})$ is a holomorphic disk in Teichm\"uller space, and if $f_t:\tilde{S}_t\to M$ is a smoothly varying family of harmonic $\rho$-equivariant maps from the universal cover $\tilde{S}_t$ of $S_t$ to $M$, we say that $((S_t, f_t):t\in\mathbb{D})$ form a complex disk of $\rho$-equivariant harmonic maps, based at $S_0$ with direction $\mu=\left.\frac{\partial S_{x+iy}}{\partial x}\right|_{x=y=0}$.
    \end{definition} 
    \begin{definition}
        If $f:\tilde{\Sigma}_g\to M$ is a $\rho$-equivariant smooth map, the pullback bundle $f^*TM$ descends to a bundle on $\Sigma_g$, that is naturally equipped with a connection. We call this bundle over $\Sigma_g$ the equivariant pullback bundle of $f$.
    \end{definition}
    We remind the reader of the notion of very strongly seminegative curvature introduced by Siu \cite{siu}, meaning that the complexification of the curvature operator is negative semi-definite. 
    \begin{theorem}\label{thm:lap-energy-zero}
        Suppose that $M$ has very strongly seminegative curvature, and let $((S_t, f_t):t\in\mathbb{D})$ be a complex disk of equivariant harmonic maps with direction $\mu$. Then $\Delta \mathrm{E}(f_t)(0)=0$ if and only if there exists a section $\xi$ of the complexified equivariant pullback bundle $\mathcal{E}^\mathbb{C}$, such that 
        \begin{gather}\label{eq:lap-energy-zero}
            \mu\partial f_0=\bar{\partial}\xi\text{ and }R^M(\xi, \partial f_0)=0.
        \end{gather}
        Moreover, in this case $\xi-\frac{\partial f}{\partial{t}}$ is a parallel section of $\ker(R^M(-, df_0))\leq \mathcal{E}$.
    \end{theorem}
    Given a complex disk of equivariant harmonic maps $((S_t, f_t):t\in\mathbb{D})$, note that the equivariant pullback bundle $\mathcal{E}$ of $S_0$ has a connection $f_0^*\nabla$ induced from the Levi--Civita connection on $TM$. Then $\left(f_0^*\nabla\right)^{0,1}$ defines a holomorphic structure on $\mathcal{E}^\mathbb{C}:=\mathcal{E}\otimes\mathbb{C}$ by the Koszul--Malgrange theorem. In the next result, we rephrase the computation of Toledo in terms of the Hodge theory of this holomorphic bundle. 
    \begin{theorem} \label{thm:lap-energy}
        Given a complex disk $((S_t, f_t):t\in\mathbb{D})$ of equivariant harmonic maps, let $\mathcal{E}$ be as above. Suppose that $\mu\partial f_0=\bar{\partial}\varphi+\theta$ is the Hodge decomposition of $\mu\partial f_0$, where $\varphi$ is a section of $\mathcal{E}^\mathbb{C}$ and $\theta$ is a $\Delta_{\bar{\partial}}$-harmonic $\mathcal{E}^\mathbb{C}$-valued $(0,1)$-form. Then 
        \begin{align*}
            \Delta \left(\mathrm{E}(f_t)\right)(0)=8\norm{\theta}_{L^2}^2+8\norm{\bar{\partial}(w-\varphi)}_{L^2}^2-8 \int_{\Sigma_g} i\langle R^M(w, \partial f_0)\wedge\bar{\partial}f_0, w\rangle,
        \end{align*}
        where $R^M$ is the Riemann curvature of $M$, $\langle\cdot,\cdot\rangle$ is the metric on $M$, both of which are extended complex linearly, and $w=\left.\frac{\partial f_t}{\partial{t}}\right|_{t=0}\in\Gamma(S, \mathcal{E}^\mathbb{C})$. 
    \end{theorem}
    From Theorem \ref{thm:lap-energy}, it easily follows that whenever $M$ has non-positive Hermitian sectional curvature, the energy $\mathrm{E}(f_t)$ is subharmonic along every complex disk of harmonic maps, which is exactly the result of Toledo \cite{Toledo2011HermitianCA}.
    \subsection{Outline and organization} In \S\ref{sect:prelim}, we recall some preliminaries about Hodge theory on holomorphic vector bundles over Riemann surfaces, and on the curvature tensor and Levi-Civita connection of symmetric spaces. We also fix some notation for the non-abelian Hodge correspondence, define the Hitchin fibration and what we mean by its smooth fibres. The rest of the paper is focused on proving our main results, that can naturally be divided into two groups. 
    \subsubsection{Riemannian manifolds} In \S\ref{sect:levi-form}, we prove Theorems \ref{thm:lap-energy-zero} and \ref{thm:lap-energy}. \par Theorem \ref{thm:lap-energy} is shown by direct computation, closely following the computation of Toledo \cite{Toledo2011HermitianCA}. The main differences are that we give a formula for $\Delta\mathrm{E}$, rather than an inequality, and that we introduce the use of Hodge theory on the complexified pullback bundle $f^*TM\otimes\mathbb{C}$, which simplifies some of the expressions.
        \par Theorem \ref{thm:lap-energy-zero} is derived from Theorem \ref{thm:lap-energy}, one direction being simple: If $\Delta\mathrm{E}=0$, the derivative $f_{{t}}$ of $f$ at $t=0$ is precisely the $\xi$ from Theorem \ref{thm:lap-energy-zero}. The converse direction is a Bochner argument that depends on the assumption of very strongly seminegative curvature. 
    \subsubsection{Higgs bundles} The rest of the paper deals with the specialized situation of completely reducible $\mathrm{GL}(n,\mathbb{C})$ representations. \par To the author's knowledge, a proof that $\mathrm{E}_\rho$ is smooth does not exist in the literature, although this was shown by Slegers \cite{slegers-smooth} to follow from the classical result of Eells--Lemaire \cite{Eells1981DeformationsOM} when $\rho$ is Hitchin. We show in \S\ref{sect:smooth-dep} that when the representation $\rho:\pi_1(\Sigma_g)\to\mathrm{GL}(n,\mathbb{C})$ is completely reducible, the harmonic map can be chosen to depend smoothly on the complex structure $S\in\Teich_g$ and on $\rho$. It immediately follows that the Higgs field and the harmonic metric also depend smoothly on $\rho, S$. This allows us to apply Theorem \ref{thm:lap-energy-zero}, and at the same time shows that $\mathrm{E}_\rho, \mathcal{R}_\rho$ are smooth.
        \par Theorem \ref{thm:higgs-lap-zero} then follows immediately from Theorem \ref{thm:lap-energy-zero}. We then prove Theorem \ref{thm:main-higgs} in \S\ref{sect:pf-main-higgs}. On the one hand, we already have a description of $K_\rho(S)$ in terms of $\mathrm{Higgs}(\rho, S)$ from Theorem \ref{thm:higgs-lap-zero}. The description of $\ker d\mathcal{R}_\rho$ in terms of the Higgs bundle $\mathrm{Higgs}(\rho, S)$ follows from a construction of the moduli space of solutions to the Hitchin equation over a varying Riemann surface, that we carry out in \S\ref{subsec:moduli-of-higgs-bundles}. In this case, both directions of the equivalence require a Bochner argument. In \S\ref{sect:hitchin-fib-proof} we show Theorem \ref{thm:hitchin-integrable-system}.
        \par In \S\ref{sect:n=1}, we analyze the case $n=1$, proving Proposition \ref{prop:n=1}. This is a straightforward corollary of Theorem \ref{thm:higgs-lap-zero}, after showing that $\phi$ from the statement of Proposition \ref{prop:n=1} is the Higgs field associated to $\rho$.
        \par In \S\ref{sect:generic-particular}, we show Theorem \ref{thm:generic-special}. From the general facts about spectral curves, it is easy to show that $K_\rho(S)$ depends only on the image of $\rho$ in the Hitchin fibration associated to $S$, as long as this fibre is smooth. We then construct in each fibre a representation $\rho$ such that $\mathrm{E}_\rho$ is strictly plurisubharmonic at $S$. This relies on the analysis of Slegers \cite{Slegers-strict}, but can also be shown easily from Theorem \ref{thm:higgs-lap-zero}. 
        \par After that, for an arbitrary $S\in\Teich_g$ for $g\geq 4$, we construct explicitly Higgs bundles over $S$ in the nilpotent cone for which the system (\ref{eq:higgs-lap-zero}) has many non-zero solutions $\mu\in T_S\Teich_g$.   
    \subsection*{Acknowledgements}
    I would like to thank Nigel Hitchin for a fruitful conversation on the relationship between equations (\ref{eq:higgs-lap-zero}) and his work in \cite{hitchin-rank-2}, that led to Theorem \ref{thm:hitchin-integrable-system}. I would also like to thank the anonymous referee for numerous comments that have improved the clarity of the paper. 
    \section{Preliminaries}\label{sect:prelim}
    \subsection{Non-abelian Hodge correspondence}\label{sect:prelim-nonabelian-hodge} We fix a marked Riemann surface $S\in\Teich_g$, and let $f:\tilde{S}\to X_n=\mathrm{GL}(n,\mathbb{C})/\mathrm{U}(n)$ be an equivariant harmonic map from its universal cover into the symmetric space associated to $\mathrm{GL}(n,\mathbb{C})$. 
    \par The pullback by $f$ of the principal $\mathrm{U}(n)$-bundle $\mathrm{GL}(n,\mathbb{C})\to X_n$ descends to a principal $\mathrm{U}(n)$-bundle $P_f$ over $S$. We consider the associated bundle $E=P_f\times_{\mathrm{U}(n)}\mathbb{C}^n$ which is a complex vector bundle equipped with a Hermitian metric. The projected Maurer--Cartan form $\omega^{\mathfrak{u}(n)}$ defines a connection $d_A$ on $E$ that preserves the metric. We will recall the definition of the Maurer--Cartan form and of the construction of this connection in \S\ref{subsubsec:maurer-cartan}. 
    \par Since $f$ is harmonic, the triple $(E, d_A^{0,1},\omega^\mathfrak{p}(\partial f))$ forms a Higgs bundle, for which the metric on $E$ is harmonic. 
    \begin{definition}\label{dfn:higgs-bundle}
        A rank $n$ Higgs bundle over a Riemann surface $S$ is a triple $(E, \bar{\partial}, \phi)$ such that $(E,\bar{\partial})$ is a holomorphic rank $n$ vector bundle over $S$, and such that $\phi$ is an $\mathrm{End}(E)$-valued holomorphic 1-form. Such a bundle is stable if all $\phi$-invariant proper subbundles have lower slope (which is degree divided by rank) than $E$. It is polystable if it is a direct sum of stable subbundles of equal slope.
    \end{definition}
    \begin{definition}
        Given a Higgs bundle $(E, \phi)$ over $S$, a Hermitian metric $h$ on $E$ is called harmonic if $\nabla^{h}+\phi+\phi^{*h}$ is flat, where $\nabla^h$ is the Chern connection on $E$.
    \end{definition}
    The non-abelian Hodge correspondence states that the map described in the first paragraph of this subsection can be inverted. 
    Denote by $\mathcal{M}_\mathrm{Higgs}^\mathrm{ps}(S)$ the set of (isomorphism classes of) polystable degree $0$ Higgs bundles over $S$, and by $\mathrm{Rep}(\pi_1(\Sigma_g),\mathrm{GL}(n,\mathbb{C}))$ the space of conjugacy classes of completely reducible representations $\rho:\pi_1(\Sigma_g)\to\mathrm{GL}(n,\mathbb{C})$.
    \begin{theorem}
        Given a Riemann surface $S$, the following map is a bijection 
        \begin{align*}
            \mathrm{Rep}(\pi_1(\Sigma_g),\mathrm{GL}(n,\mathbb{C}))&\longrightarrow \mathcal{M}_\mathrm{Higgs}^\mathrm{ps}(S)\\
            \rho&\longrightarrow \mathrm{Higgs}(\rho, S):=(E, d_A^{0,1}, \omega^\mathfrak{p}(\partial f)),
        \end{align*}
        where $f:\tilde{S}\to X$ is the a $\rho$-equivariant harmonic map.
    \end{theorem} 
    \subsubsection{The Hitchin fibration}\label{subsubsec:hitchin-fibration}
   Given a Riemann surface $S$, let $K_S$ be its cotangent bundle. We denote the Hitchin fibration by 
   \begin{align*}
    H:\mathcal{M}_\mathrm{Higgs}^\mathrm{ps}(S)&\longrightarrow\mathcal{B}(S)=\bigoplus_{i=1}^n H^0(S, K_S^i)\\
    (E,\phi)&\longrightarrow (p_1(\phi), p_2(\phi),...,p_n(\phi)).
   \end{align*}
   Here $p_i(\phi)$ form a basis for the space of $\mathrm{GL}(n,\mathbb{C})$ invariant polynomials on $\mathfrak{gl}(n,\mathbb{C})$, so that $p_i$ is homogeneous of degree $i$. A natural choice for $p_i(\phi)$ is the  $x^{n-i}$ coefficient in the characteristic polynomial $\chi_\phi(x)=\det(x\mathrm{id}_E-\phi)$.
    \par Let $(E,\phi)$ be a Higgs bundle, $a=H(E,\phi)$, and consider the subvariety $S_a$ of the total space of the canonical bundle $K_S$ cut out by the equation 
    \begin{align*}
        \chi_a(x)=x^n+x^{n-1}p_1(\phi)+...+p_n(\phi)=0,
    \end{align*}
    where $\chi_a(x)=\det(x\mathrm{id}_E-\phi)$ is the characteristic polynomial of $\phi$ (that depends only on $a$). This subvariety is called the spectral curve.
    \begin{definition} The point $a\in\bigoplus_{i=1}^n H^0(S, K_S^i)$ defines a smooth fibre $H^{-1}(a)$ of the Hitchin fibration when $S_a$ is reduced, irreducible, and smooth.
    \end{definition}
    When $g\geq 3, n\geq 2$, for a generic $a\in\bigoplus_{i=1}^n H^0(S, K_S^i)$, it is well-known that the curve $S_a$ is smooth. This follows e.g. from \cite[Proposition 2.1]{markman} and the fact that $\deg(K_S^n)\geq 4g-4>2g+1$, and hence $K_S^n$ is very ample.
    \subsection{Hodge theory on holomorphic vector bundles}Let $E$ be a complex hermitian vector bundle over a Riemann surface $S$. If $E$ is equipped with a connection $\nabla$, the $(0,1)$-part of this connection defines the structure of a holomorphic vector bundle on $E$. If $\nabla$ is unitary for the metric on $E$, then it is equal to the Chern connection on $E$. 
    \par We write $\nabla=\partial+\bar{\partial}$ for the splitting of $\nabla$ into its $(1,0)$ and $(0,1)$ parts. If $S$ is equipped with a volume form, that is automatically K\"ahler, we may construct the formal adjoint $\nabla^*$ of $\nabla$. We split $\nabla^*=\partial^*+\bar{\partial}^*$ into its $(1,0)$ and $(0,1)$ parts. We may now construct holomorphic and antiholomorphic Laplacians on $E$-valued differential forms, 
    \begin{gather*}
        \Delta_{\bar{\partial}}=\bar{\partial}^*\bar{\partial}+\bar{\partial}\bar{\partial}^*,\\
        \Delta_{\partial}=\partial^*\partial+\partial\partial^*.
    \end{gather*}
    It is well-known that these are elliptic, and satisfy a Bochner--Kodaira--Nakano identity \cite[Theorem (VII.1.2)]{demailly1997complex}
    \begin{align*}
        \Delta_{\bar{\partial}}=\Delta_\partial+i[F_\nabla, \Lambda_\omega],
    \end{align*}
    where $F_\nabla$ is the curvature of $\nabla$, and $\Lambda_\omega$ is contraction with the K\"ahler form $\omega$. In particular, $\Delta_\partial$ and $\Delta_{\bar{\partial}}$ agree on 1-forms. Standard Hodge theory now shows the following. 
    \begin{proposition}\label{prop:hodge-theory}
        Let $\xi$ be an $E$-valued $(0,1)$ form. Then there exists a section $g$ of $E$, and a closed and coclosed $E$-valued $(0,1)$ form $\theta$, such that $\xi=\bar{\partial}g+\theta$. Moreover, when $(E, \nabla)$ comes from the complexification of a real vector bundle, then $\theta$ is complex conjugate to a holomorphic $E$-valued 1-form.
    \end{proposition}
    \subsection{Symmetric spaces}
    Our aim in this subsection is to recall some general theory on symmetric spaces that will be useful in the sequel. 
    \par The most important result for us is that symmetric spaces of noncompact type have very strongly seminegative curvature, in the sense of Siu \cite{siu}. We first recall the definition of Siu. 
    \begin{definition}A manifold $X$ has very strongly seminegative curvature if the sesquilinear form defined on $\bigwedge^2 TX\otimes\mathbb{C}$ by 
        \begin{align*}
            Q(\alpha\wedge \beta, \gamma\wedge \delta)=\langle R(\alpha,\beta)\bar{\delta},\bar{\gamma}\rangle
        \end{align*}
        is negative semidefinite, where $R$ is the Riemann curvature of $X$.
    \end{definition}
    The following proposition appears in the report of Loustau \cite[Corollary 5.5]{loustau}.
    \begin{proposition}\label{prop:symm-spaces-very-strongly-negative-curv}
        Let $X$ be a locally symmetric space of non-positive curvature. Then $X$ has very strongly seminegative curvature.
    \end{proposition} 
    
     We also recall some theory about the Maurer--Cartan form and how it relates to the tangent bundle of $X$ and its Levi--Civita connection. 
    \subsubsection{Maurer--Cartan form}\label{subsubsec:maurer-cartan} Let $X=G/K$ equipped with some left-invariant metric be a globally symmetric space, where $G$ is a Lie group with a compact subgroup $K$. In this subsection we identify $TX$ with a $G\times_K\mathfrak{p}$, and describe the Levi--Civita connection in terms of this identification. We assume throughout that $X$ has non-positive curvature, and adopt the notation of $\mathfrak{g}, \mathfrak{k}, \mathfrak{p}, \mathfrak{g}^\mathrm{nc}, \mathfrak{k}^\mathrm{nc}, \mathfrak{p}^\mathrm{nc}, \mathfrak{g}^\mathrm{e}, \mathfrak{k}^\mathrm{e}, \mathfrak{p}^\mathrm{e}$ as before.
    \par Let $\omega\in\Omega^1(G, \mathfrak{g})$ be the Maurer--Cartan form of $G$. We will denote by $\omega^{\mathfrak{k}},\omega^\mathfrak{p}$ the results of postcomposing $\omega$ with the projections $\mathfrak{g}\to\mathfrak{k}, \mathfrak{g}\to\mathfrak{p}$, respectively, given by the splitting $\mathfrak{g}=\mathfrak{k}+\mathfrak{p}$. The form $\omega$ has the following invariance properties 
    \begin{enumerate}
        \item $L_g^*\omega=\omega$, and 
        \item $R_g^*\omega=\mathrm{Ad}_{g^{-1}}\omega$.
    \end{enumerate}
    In particular, $\omega$ identifies all tangent spaces of $G$ to $\mathfrak{g}$, in a left-invariant manner.
    \par The form $\omega^\mathfrak{p}$ descends to a form on $X=G/K$, taking values in the bundle $G\times_K \mathfrak{p}$ over $X$, where $K$ acts on $\mathfrak{p}$ by the restriction of the adjoint action of $G$. It defines an isomorphism $TX\to G\times_K\mathfrak{p}$. Similarly, the form $\omega^\mathfrak{k}$ defines a connection on the principal $K$-bundle $G\to X$. There is a connection $\nabla$ on $G\times_K\mathfrak{p}$ induced from $\omega^\mathfrak{k}$. The following result is well-known in the theory of symmetric spaces. Since an explicit reference is unknown to the author, we include a proof for completeness. Note that in a slightly less general setting, the same result was shown by Slegers \cite[Lemma 2.2]{Slegers-strict}.
    \begin{proposition}\label{prop:levi-civita-on-symm-space}
        If $V$ is a vector field on $X$, then $\omega^\mathfrak{p}(\nabla V)=\nabla\left( \omega^\mathfrak{p}(V)\right)$. 
    \end{proposition}
\begin{proof}
    We equip $G$ with a left-invariant metric such that $G\to G/K$ is a Riemannian submersion. We denote by $\nabla$ the Levi-Civita connection on $TG$. 
    \begin{claim}
        Let $V$ be a vector field on $G$. Then $\omega(\nabla V)=d(\omega(V))+\frac{1}{2}[\omega, \omega(V)]$.
    \end{claim}
    \begin{proof}
        Note that by the Koszul formula, if $\xi,\eta$ are left-invariant vector fields on $G$, we have 
        \begin{align*}
            \nabla_\xi \eta=\frac{1}{2}[\xi, \eta].
        \end{align*}
        Since $\eta,\xi$ are left-invariant, $\omega(\xi), \omega(\eta)$ are constant as functions $G\to\mathfrak{g}$. We denote their values by $\xi, \eta$, respectively. 
        We now have  
        \begin{align*}
            \nabla_\xi \langle \omega(V), \eta\rangle =\nabla_\xi \langle V, \eta\rangle&=\langle \nabla_\xi V, \eta\rangle+\left\langle V, \frac{1}{2}[\xi, \eta]\right\rangle.
        \end{align*}
        Thus 
        \begin{align*}
            \langle d_\xi (\omega(V)), \eta\rangle&=\langle \omega(\nabla_\xi V), \eta\rangle+\frac{1}{2}\langle \omega(V), [\xi, \eta] \rangle\\
            &=\langle \omega(\nabla_\xi V), \eta\rangle+\frac{1}{2}\langle [\omega(V), \xi], \eta\rangle.
        \end{align*}
        Since $\eta\in\mathfrak{g}$ is arbitrary, we see that $d_\xi\omega(V)=\omega(\nabla_\xi V)+\frac{1}{2}[\omega(V), \omega(\xi)]$. Thus 
        \begin{align*}
            \omega(\nabla V)=d(\omega(V))-\frac{1}{2}[\omega(V), \omega]
        \end{align*} 
        as desired.
    \end{proof}
    Let $W$ be a vector field on $X$ such that $\omega^\mathfrak{p}(W)=V$. We let $\bar{W}$ be the vector field on $G$ such that $\omega(\bar{W})=V$. Then $\bar{W}$ projects to $W$ under the natural quotient map $G\to X$. Since $\omega(\bar{W}(g))\in\mathfrak{p}$, we see that $\bar{W}(g)\in (L_g)_\star\mathfrak{p}$. In particular, $\bar{W}$ is perpendicular to the fibres of $G\to X$.
    \par Given any section $V'$ of $G\times_K\mathfrak{p}$, introduce vector fields $W', \bar{W}'$ analogously. It follows by standard results on Riemannian submersions \cite[Proposition 13 (4)]{petersen2006riemannian}, that $\nabla_{W'} W$ is the projection of $\nabla_{\bar{W}'}\bar{W}$ to $X$. Thus, using $[\mathfrak{p},\mathfrak{p}]\leq\mathfrak{k}$, 
    \begin{align*}
        \omega^\mathfrak{p}(\nabla_{W'} W)&=\omega^\mathfrak{p}(\nabla_{\bar{W}'} \bar{W})=d_{\bar{W}'} V+\frac{1}{2}[V',V]^\mathfrak{p}=d_{\bar{W}'}V.
    \end{align*}
    The result now follows since $\omega^\mathfrak{k}(\bar{W}')=0$. 
\end{proof}
    \section{Levi form of the energy}\label{sect:levi-form}
    Let $((S_t, f_t):t\in\mathbb{D})$ be a complex disk of harmonic maps for some representation $\rho:\pi_1(\Sigma_g)\to\mathrm{Isom}(M)$. We then have 
    \begin{align*}
        \mathrm{E}(f_t)=i\int_{\Sigma_g} \langle\partial f_t\wedge \bar{\partial}f_t\rangle.
    \end{align*}
    In this section, we prove Theorems \ref{thm:lap-energy-zero} and \ref{thm:lap-energy}.
    \par As mentioned in the outline, deriving Theorem \ref{thm:lap-energy-zero} from Theorem \ref{thm:lap-energy} consists of two steps: showing that if $\Delta\mathrm{E}=0$, then the system (\ref{eq:lap-energy-zero}) has a solution, which follows immediately from Theorem \ref{thm:lap-energy}, and showing the converse, that depends on a Bochner argument that shows that any solution to (\ref{eq:lap-energy-zero}) has to differ from $\frac{\partial f}{\partial t}$ by a parallel section of $\ker R^M(-, df_0)$. This Bochner argument relies on the assumption of very strongly seminegative curvature for $M$, and on second order elliptic equations for $\frac{\partial f}{\partial t}$ obtained in \S\ref{subsec:first-order-var} by taking the first-order variation of the harmonic map equation. Using this argument, we prove Theorem \ref{thm:lap-energy-zero} assuming Theorem \ref{thm:lap-energy} in \S\ref{subsec:pf-thm-lap-energy-zero-assuming-thm-lap-energy}.
    \par In \S\ref{subsec:pf-thm-lap-energy}, we prove Theorem \ref{thm:lap-energy}, following \cite{Toledo2011HermitianCA}. There are two major differences between our computation and that in \cite{Toledo2011HermitianCA}. First, we give an actual equality, rather than an inequality as in \cite{Toledo2011HermitianCA}, and give a slightly more precise analysis of the equality case than \cite{Toledo2011HermitianCA}. Second, by using Hodge theory on the complexified equivariant pullback bundle $f_0^*TM$, we simplify some of the expressions. 
    \subsection{First-order variation of the harmonic map equation}\label{subsec:first-order-var}
    Suppose first that $((f_t, J_t):t\in(-1,1))$ is an open interval of $\rho$-equivariant harmonic maps. We assemble them into a map 
    \begin{align*}
        F:\Sigma_g\times(-1,1)\to M.
    \end{align*}
    Let $f=f_0$. We now consider the vector bundle $(F^*TM, F^*\nabla)$. Let $\Pi_t:f^*TM\to F(-,t)^*TM$ be the parallel transport in this bundle along vertical paths of the form $\{-\}\times [0,t]$ or $\{-\}\times[t,0]$. We let $\nabla^t=\Pi_t^*(F^*\nabla)$. We let $\mathcal{E}$ be the equivariant pullback bundle of $f_0$, and $\mathcal{E}^\mathbb{C}=\mathcal{E}\otimes\mathbb{C}$.
    \begin{proposition}\label{prop:variation-of-harmonic-eq}
        For $\dot{f}=\frac{\partial F}{\partial t}$, we have $\dot{\nabla}=R^M(\dot{f}, df)$, where $R^M$ is the curvature tensor of $M$, and 
        \begin{gather*}
            d^\nabla\left(\partial^\nabla\dot{f}+\mu\partial f-\bar{\mu}\bar{\partial}f\right)+R^M(\dot{f}, \bar{\partial}f)\wedge\partial f=0,\\
            d^\nabla\left(\bar{\partial}^\nabla\dot{f}-\mu\partial f+\bar{\mu}\bar{\partial}f\right)+R^M(\dot{f}, \partial f)\wedge\bar{\partial}f=0.
        \end{gather*}
    \end{proposition}
    \begin{proof}
        If $X$ is a vector field on $\Sigma_g$, and $s$ is a section of $\mathcal{E}$, we have 
        \begin{align*}
            \left.\frac{d}{dt}\right|_{t=0}\nabla^t_X s&=\nabla_{\partial_t}|_{t=0}\nabla_X (\Pi_t s)\\
            &=\nabla_X(\nabla_{\partial_t}\Pi_t s)+R^{F^*TM}(\partial_t, X)s\\
            &=R^M(\dot{f}, f_*X)s,
        \end{align*}
        where $R^{F^*TM}$ denotes the curvature of the bundle $(F^*TM, F^*\nabla)$. 
        \par We now turn to the second claim. We only show the first equation, since the second one is completely analogous. The harmonic map equation is 
        \begin{align*}
            d^{\nabla^t} \left(\left(\Pi_t^{-1}df_t\right)\circ\frac{\mathrm
            {id}-iJ_t}{2}\right)=0.
        \end{align*}
        Since $\Pi_t^{-1}df_t=df+t d^\nabla\dot{f}+O(t^2)$, we have by differentiating and using \cite[Claim 3.2]{mt-hodge}, 
        \begin{align*}
            R^M(\dot{f}, df)\wedge \left(df\circ\frac{\mathrm
            id-iJ}{2}\right)+d^\nabla\left(d^\nabla \dot{f}\circ \frac{\mathrm{id}-iJ}{2}\right)-\frac{i}{2}d^\nabla\left(2i\mu\partial{f}-2i\bar{\mu}\bar{\partial}f\right)=0.
        \end{align*}
        Therefore 
        \begin{align*}
            d^\nabla\left(\partial^\nabla\dot{f}+\mu\partial f-\bar{\mu}\bar{\partial}f\right)+R^M(\dot{f}, \bar{\partial}f)\wedge\partial f=0. 
        \end{align*}
    \end{proof}
    Now note that $\nabla^{0,1}$ makes $\mathcal{E}^\mathbb{C}$ into a holomorphic vector bundle over $(\Sigma_g, J_0)$. Therefore we can write 
    \begin{align*}
        \mu\partial f=\bar{\partial}^\nabla \varphi+\overline{\theta},
    \end{align*} 
    where $\varphi$ is a section of $\mathcal{E}^\mathbb{C}$ and $\theta$ is a holomorphic $\mathcal{E}^\mathbb{C}$-valued 1-form. For brevity of notation, from now on we drop the reference to the connection and write $d,\partial,\bar{\partial}$ for $d^\nabla, \partial^\nabla, \bar{\partial}^\nabla$.
    \begin{corollary}
        In the notation given above, 
        \begin{gather*}
            \bar{\partial}\partial(\dot{f}-2\mathrm{Re}(\varphi))=-R^M(\dot{f}, \bar{\partial}f)\wedge\partial f-f^*R^M\varphi,\\
            \partial\bar{\partial}(\dot{f}-2\mathrm{Re}(\varphi))=-R^M(\dot{f}, \partial f)\wedge\bar{\partial}f-f^*R^M\bar{\varphi}.
        \end{gather*}
    \end{corollary}
    \begin{proof}
        Again we only show the first equation. Note that 
        \begin{align*}
            \bar{\partial}\partial(\dot{f}-\varphi-\bar{\varphi})&=d^\nabla(\partial\dot{f}+\bar{\partial}\varphi-\partial\bar{\varphi})-(\partial\bar{\partial}+\bar{\partial}\partial)\varphi\\
            &=-R^M(\dot{f},\bar{\partial}f)\wedge\partial f-(d^\nabla)^2\varphi, 
        \end{align*}
        which concludes the proof by definition of the curvature tensor. 
    \end{proof}
    \subsection{Proof of Theorem \ref{thm:lap-energy-zero} assuming Theorem \ref{thm:lap-energy}}\label{subsec:pf-thm-lap-energy-zero-assuming-thm-lap-energy}
    By Theorem \ref{thm:lap-energy}, $\Delta \mathrm{E}(f_t)=0$ if and only if for some $\varphi\in \Gamma(\mathcal{E}^\mathbb{C})$, we have 
    \begin{gather}
        \mu\partial f=\bar{\partial}\varphi,\nonumber\\
        \bar{\partial}(W-2\varphi)=0,\nonumber\\
        \langle R^M(W,\partial f)\wedge\bar{\partial}f, \bar{W}\rangle=0,\label{eq:curv-equality}
    \end{gather}
    where $W=2w$. Using elementary linear algebra, we can rewrite (\ref{eq:curv-equality}) as follows.
    \begin{claim}\label{claim:curv-la}
        For any two $X,Y\in TM\otimes\mathbb{C}$, the equality $\langle R^M(X,Y)\bar{Y},\bar{X}\rangle=0$ holds if and only if $R^M(X,Y)=0$. 
    \end{claim}
    \begin{proof}
        We define the sesquilinear form $Q$ on $\bigwedge^2 TM\otimes\mathbb{C}$ by 
        \begin{align*}
            Q(X\wedge Y, Z\wedge W)=\langle R^M(X, Y)\bar{W},\bar{Z}\rangle. 
        \end{align*}
        This is well-defined by the standard symmetries of the Riemann curvature $R^M$. Since $M$ has very strongly seminegative curvature, by definition $Q$ is negative semi-definite. Thus $Q(X\wedge Y, X\wedge Y)=0$ if and only if $Q(X\wedge Y, -)=0$, which is in turn equivalent to $R^M(X,Y)=0$. 
    \end{proof}
    Therefore (\ref{eq:curv-equality}) is equivalent to 
    \begin{align}\label{eq:curv-equality'}
        R^M(W,\partial f)=0.
    \end{align}
    Hence if $\Delta \mathrm{E}(f_t)=0$, the required solution is $\xi=\frac{1}{2}W$.
    \par Conversely, assume that for a section $\xi$ of $\mathcal{E}^\mathbb{C}$, we have 
    \begin{align*}
        \mu\partial f=\bar{\partial}\xi\text{ and }R^M(\xi,\partial f)=0.
    \end{align*}
    Note that by Proposition \ref{prop:variation-of-harmonic-eq}, we have 
    \begin{gather*}
        d\left(\bar{\partial}\dot{f}^\mu-\mu\partial f+\bar{\mu}\bar{\partial}f\right)+R^M(\dot{f}^\mu,\partial f)\wedge \bar{\partial}f=0,\\
        d\left(\bar{\partial}\left(i\dot{f}^{i\mu}\right)+\mu\partial f+\bar{\mu}\bar{\partial}f\right)+R^M(i\dot{f}^{i\mu},\partial f)\wedge \bar{\partial}f=0.
    \end{gather*}
    Subtracting, and recalling that $W=\dot{f}^\mu-i\dot{f}^{i\mu}$, we have 
    \begin{align*}
        \partial\bar{\partial}(W-2\xi)+R^M(W,\partial f)\wedge\bar{\partial}f=0.
    \end{align*}
    Since $R^M(\xi,\partial f)=0$, we have for $V=W-2\xi$, the equation 
    \begin{align*}
        \partial\bar{\partial}V+R^M(V,\partial f)\wedge\bar{\partial}f=0.
    \end{align*}
    Taking the inner product with $\bar{V}$ and integrating, we have 
    \begin{align*}
        \int_{\Sigma_g} \frac{i}{2}\langle R^M(V,\partial f)\wedge{\bar{\partial}f}, \bar{V}\rangle&=\frac{i}{2}\int_{\Sigma_g} -\langle \bar{V}, \partial\bar{\partial}V\rangle \\
        &=\frac{i}{2}\int_{\Sigma_g} d\langle\bar{V}, \bar{\partial}V\rangle-\langle \bar{V}, \partial\bar{\partial}V\rangle\\
        &=\frac{i}{2}\int_{\Sigma_g}\langle \partial\bar{V}\wedge\bar{\partial}V \rangle=\norm{\bar{\partial}V}_{L^2}^2.
    \end{align*}
    Since $M$ has very strongly seminegative curvature, it also has non-positive Hermitian sectional curvature. Therefore $\frac{i}{2}\langle R^M(V,\partial f)\wedge\bar{\partial}f, \bar{V}\rangle \leq 0$. Therefore we must have 
    \begin{align}\label{eq:v-separated}
        \bar{\partial}V=0\text{ and }\langle R^M(V,\partial f)\wedge\bar{\partial}f, \bar{V}\rangle=0.
    \end{align}
    Thus $\bar{\partial}W=2\bar{\partial}\xi$, and hence $\Delta \mathrm{E}(f_t)=0$. 
    \par We now turn to the final statement. From (\ref{eq:v-separated}) and Claim \ref{claim:curv-la}, it follows that $\bar{\partial}V=0$ and $R^M(V, \partial f)=0$. We aim to show that $\partial V=0$ as well. This follows from a Bochner-type computation 
    \begin{align*}
        \norm{\partial V}_{L^2}^2&=\frac{i}{2}\int_{\Sigma_g}\langle \partial V\wedge\bar{\partial}\bar{V}\rangle=\frac{i}{2}\int_{\Sigma_g} d\langle V, \bar{\partial}\bar{V}\rangle-\langle V, \partial\bar{\partial}\bar{V}\rangle\\
        &=-\frac{i}{2}\int_{\Sigma_g} \langle V, -\bar{\partial}\partial \bar{V}+f^*R^M\bar{V}\rangle=-\frac{i}{2}\int_{\Sigma_g}\langle V, f^*R^M\bar{V}\rangle.
    \end{align*}
    We have, in a local holomorphic coordinate $z$ on $(\Sigma_g, J)$, using the Bianchi identity
    \begin{align*}
        f^*R^M\bar{V}&=R^M(f_z, f_{\bar{z}})\bar{V} dz\wedge d\bar{z}=-(R^M(\bar{V}, f_z)f_{\bar{z}}-R^M(f_{\bar{z}}, \bar{V})f_z)dz\wedge d\bar{z}\\ &=-R^M(\bar{V}, f_z)f_{\bar{z}}dz\wedge d\bar{z}.
    \end{align*}
    Therefore 
    \begin{align*}
        -\frac{i}{2}\langle V, f^*R^M\bar{V}\rangle&=\frac{i}{2} dz\wedge d\bar{z}\langle R^M(\bar{V}, f_z)f_{\bar{z}}, V\rangle\leq 0,
    \end{align*}
    since $M$ has non-positive Hermitian sectional curvature. Thus $\norm{\partial V}_{L^2}^2\leq 0$, so $\partial V=0$ and $R^M(V, {f_{\bar{z}}})=0$ by Claim \ref{claim:curv-la}. Along with $\bar{\partial}V=0$ and $R^M(V, \partial f)=0$, this implies that $dV=0$ and $R^M(V, df)=0$, as desired.
    \subsection{Proof of Theorem \ref{thm:lap-energy}}\label{subsec:pf-thm-lap-energy}
    We first compute some formulas for the second variation of the energy in a direction defined by $\mu$. Thus we are still in the setting where $((J_t, f_t):t\in(-1,1))$ is an interval of equivariant harmonic maps. As in the previous section set $f=f_0$ and $J=J_0$, and equip $\mathcal{E}^\mathbb{C}$ with the holomorphic structure coming from $(f^*\nabla)^{0,1}$, and write $\mu\partial f=\bar{\partial}\varphi+\bar{\theta}$. 
    \par Note that 
    \begin{align*}
        \mathrm{E}(f_t)=i\int_{\Sigma_g} \left\langle df_t\circ \frac{\mathrm{id}-iJ_t}{2}\wedge df_t\circ\frac{\mathrm{id}+iJ_t}{2} \right\rangle=-\frac{1}{2}\int_{\Sigma_g}\langle df_t\wedge df_t\circ J_t\rangle.
    \end{align*}
    We now define $F(s,t)=-\frac{1}{2}\int_{\Sigma_g} \langle df_{s}\wedge df_s\circ J_t\rangle$. Since $f_t$ is harmonic for $J_t$, we have $\frac{\partial F}{\partial s}(t,t)=0$. Therefore 
    \begin{align}
        \left.\frac{d^2}{dt^2}\right|_{t=0}F(t,t)&=\left(\frac{\partial}{\partial s}+\frac{\partial}{\partial t}\right)^2_{t=0} F(t,t)=\left(\frac{\partial}{\partial s}+\frac{\partial}{\partial t}\right) \frac{\partial F}{\partial t}(0,0)\nonumber\\
        &=\frac{\partial^2 F}{\partial t^2}(0,0)+\frac{\partial^2 F}{\partial t\partial s}(0,0).\label{eq:two-terms-energy-lap}
    \end{align}
    We observe that since parallel transport preserves the metric, we have 
    \begin{align*}
        F(s,t)=-\frac{1}{2}\int_{\Sigma_g} \langle \Pi_s^{-1}df_s\wedge \Pi_s^{-1}df_s\circ J_t\rangle. 
    \end{align*}
    From now on, we transport all derivatives $df_s$ to the bundle $\mathcal{E}$ by $\Pi_s$, and work exclusively on $\mathcal{E}$. 
    \par We will compute the two terms in (\ref{eq:two-terms-energy-lap}) separately, but first we recall the first two derivatives of $J$ in the direction $\mu$. This is essentially \cite[Claim 3.2]{mt-hodge}.
    \begin{claim}\label{claim:variation-j}
        We have $\ddot{J}=4\abs{\mu}^2 J$, and for any 1-form $\omega$, 
        \begin{align*}
            \omega\circ\dot{J}=2i\left(\mu\omega^{1,0}-\bar{\mu}\omega^{0,1}\right).
        \end{align*}
    \end{claim} 
    \subsubsection{First term} We have from Claim \ref{claim:variation-j}, 
    \begin{align}
        \frac{\partial^2 F}{\partial t^2}(0,0)&=-\frac{1}{2}\int_{\Sigma_g} 4\abs{\mu}^2\langle df\wedge df\circ J\rangle=4 i\int_{\Sigma_g} \abs{\mu}^2\langle \partial f\wedge \bar{\partial}f\rangle \nonumber\\
        &=4i\int_{\Sigma_g}\langle\overline{\mu\partial f}\wedge\mu\partial f \rangle=4i\int_{\Sigma_g} \langle  (\partial\bar{\varphi}+\theta)\wedge(\bar{\partial}\varphi+\bar{\theta})\rangle \nonumber\\
        &=4 \norm{\bar{\partial}\varphi}^2+4\norm{\theta}^2.\label{eq:ft-final}
    \end{align}
    Here we introduced the $L^2$ norm for $(1,0)$- or $(0,1)$-forms, by 
    \begin{align*}
        \norm{\phi}^2=i\int_{\Sigma_g} \langle \phi\wedge\bar{\phi}\rangle, 
    \end{align*} 
    for $(1,0)$-forms $\phi$, and $\norm{\bar{\phi}}=\norm{\phi}$. We denote the associated Hermitian inner product by $(\cdot, \cdot)_{L^2}$.
    \subsubsection{Second term} We have
    \begin{align}
        \frac{1}{2}\frac{\partial^2 F}{\partial s\partial t}(0,0)&=-\frac{1}{2}\int_{\Sigma_g} \langle d^\nabla\dot{f}\wedge df\circ \dot{J}\rangle=-i\int_{\Sigma_g} \langle d^\nabla\dot{f}\wedge\left( \mu \partial f-\bar{\mu}\bar{\partial}f\right)\rangle\nonumber\\
        &=-i\int_{\Sigma_g} \langle d^\nabla\dot{f}\wedge (\bar{\partial}\varphi+\bar{\theta}-\partial\bar{\varphi}-\theta)\rangle=-i\int_{\Sigma_g} \langle \partial\dot{f}\wedge\bar{\partial}\varphi\rangle-\overline{\langle \partial\dot{f}\wedge \bar{\partial}\varphi\rangle}\nonumber\\
        &=2\mathrm{Im}\int_{\Sigma_g}\langle \partial\dot{f}\wedge\bar{\partial}\varphi\rangle=-({\partial}{\dot{f}}, {\partial}\bar{\varphi})_{L^2}-({\partial}\bar{\varphi}, {\partial}\dot{f})_{L^2}. \label{eq:st-first}  
    \end{align}
    We now compute this term in a different way, 
    \begin{align}
        \frac{1}{2}\frac{\partial^2 F}{\partial s\partial t}(0,0)&=-\frac{1}{2}\frac{\partial^2 F}{\partial s^2}=-\frac{1}{2}\int_{\Sigma_g} \langle \dot{f}, \mathcal{J}\dot{f}\rangle d\mathrm{area}_{g_0},\label{eq:st-second}
    \end{align}
    where $\mathcal{J}$ is the Jacobi operator for some background conformal metric $g_0$ on $(\Sigma_g, J)$, given by 
    \begin{align*}
        \mathcal{J}V=-\Delta V+ \sum_{i} R^M(\nabla_{e_i} f, V)\nabla_{e_i} f,
    \end{align*} 
    where $e_i$ form a fibrewise orthonormal basis of $f^*TM$. We will review some background on the Jacobi operator in \S\ref{subsec:jacobi-operator}. By Proposition \ref{prop:micaleff-moore}, we have 
    \begin{align*}
        \mathcal{J}V d\mathrm{area}_{g_0}&=-2i\left(\partial\bar{\partial}V+R(V, \partial f)\wedge\bar{\partial}f\right).
    \end{align*}
    \par We now return to the setting of a complex disk of equivariant harmonic maps $((J_z, f_z):z\in\mathbb{D})$, where we have renamed the variable to $z=x+iy$ to avoid confusion. We are interested in $\Delta \mathrm{E}(f_z)=\frac{\partial^2\mathrm{E}}{\partial x^2}+\frac{\partial^2\mathrm{E}}{\partial y^2}$. Write $W=2\frac{\partial f}{\partial z}=\dot{f}^\mu-i\dot{f}^{i\mu}$. We analogously introduce variables $t_1, s_1$ associated to $x$, and $t_2,s_2$ associated to $y$ as above. Then by (\ref{eq:st-first}),
    \begin{align*}
        \frac{1}{2}\left(\frac{\partial^2 F}{\partial s_1\partial t_1}+\frac{\partial^2F}{\partial s_2\partial t_2}\right)&=-(\partial \bar{W}, \partial\bar{\varphi})_{L^2}-(\partial\bar{\varphi}, \partial \bar{W})_{L^2},
    \end{align*}
    and by (\ref{eq:st-second}), 
    \begin{align*}
        \frac{1}{2}\left(\frac{\partial^2 F}{\partial s_1\partial t_1}+\frac{\partial^2F}{\partial s_2\partial t_2}\right)&=-\frac{1}{2}\int_{\Sigma_g}
    \left( \langle \dot{f}^{\mu}, \mathcal{J}\dot{f}^\mu\rangle+\langle \dot{f}^{i\mu}, \mathcal{J}\dot{f}^{i\mu}\rangle\right) d\mathrm{area}_{g_0}\\
        &=-\frac{1}{4}\int_{\Sigma_g}\left(\langle W, \mathcal{J}\bar{W}\rangle+\langle \bar{W}, \mathcal{J}W\rangle\right)d\mathrm{area}_{g_0}\\
        &=-\frac{1}{2}\mathrm{Re}\int_{\Sigma_g} \langle \bar{W}, \mathcal{J}{W}\rangle d\mathrm{area}_{g_0}\\
        &=\mathrm{Re}\left(i\int_{\Sigma_g} \langle \bar{W}, \partial\bar{\partial}{W}\rangle+\langle\bar{W}, R(W, \partial f)\wedge\bar{\partial}f\rangle\right)\\
        &=-\norm{\bar{\partial}W}^2+i\int_{\Sigma_g}\langle R(W, \partial f)\wedge\bar{\partial}f,\bar{W}\rangle.
    \end{align*}
    Therefore using these two equalities, we have 
    \begin{align*}
        \Delta E(0)&=\mathrm{I}+\mathrm{II}, 
    \end{align*}
    where 
    \begin{gather*}
        \mathrm{I}=8\norm{\bar{\partial}\varphi}^2+8\norm{\theta}^2,\\
        \mathrm{II}=-2(\partial \bar{W}, \partial\bar{\varphi})_{L^2}-2,(\partial\bar{\varphi}, \partial\bar{W})_{L^2}=-2\norm{\bar{\partial}W}^2+2\mathcal{R}\\
    \mathcal{R}=i\int_{\Sigma_g}\langle R(W, \partial f)\wedge\bar{\partial}f, \bar{W}\rangle.
   \end{gather*}
   In particular, we have 
   \begin{align*}
    \mathrm{II}&=2\left(-2(\partial\bar{W},\partial\bar{\varphi})_{L^2}-2\left(\partial\bar{\varphi},\partial\bar{W}\right)_{L^2}\right)+2\norm{\bar{\partial}W}^2-2\mathcal{R}\\
    &=2\norm{\bar{\partial}(W-2\varphi)}^2-8\norm{\bar{\partial}\varphi}^2-2\mathcal{R}.
   \end{align*}
   This concludes the proof of Theorem \ref{thm:lap-energy}.
   \section{Smooth dependence on the representation and complex structure}\label{sect:smooth-dep}
   This section is devoted to showing that $\mathrm{E}_\rho$ and $\mathcal{R}_\rho$ are smooth, and having sufficient machinery to be able to compute their derivatives. 
   \par To show smoothness of $\mathrm{E}_\rho$, we show that the $\rho$-equivariant harmonic map $f:(\tilde{\Sigma}_g, J)\to X_n$ can be chosen to depend smoothly on $\rho$ and $J$. We show this by using the Banach manifold implicit function theorem, analogously to the classical result of Eells--Lemaire \cite{Eells1981DeformationsOM}. Their argument essentially already shows that when $f$ is unique, it depends smoothly on $J$. However, since we are interested in showing smooth dependence on the representation as well, we work in the setting of harmonic metrics on flat bundles. Note that here we show smoothness on the representation itself, not on the corresponding element of the character variety. In \S\ref{subsec:dictionary-harmonic-map-harmonic-metric}, we explain how the equation for the harmonic metric on a flat bundle is equivalent to the harmonic map equation, and how its first variation is the Jacobi operator of the associated harmonic map. We will use some standard results on the Jacobi operator in the sequel, so we recall them in \S\ref{subsec:jacobi-operator}. In \S\ref{subsec:main-smoothness-result}, we show that $f$ can be chosen to depend smoothly on $\rho, J$. 
   \par It now follows immediately that $\mathcal{R}$ is smooth, since the Higgs bundle $(E,\phi)$ and the harmonic metric $h$ depend smoothly on $\rho, J$, so does the flat connection associated to $(E,i\phi)$. However, in proving Theorem \ref{thm:main-higgs}, it will be convenient to be able to say that $(\mathcal{R}_\rho)_*\mu=0$ if and only if the associated solution to the Hitchin equation does not change to first order. For this we need to construct the moduli space of solutions to the Hitchin equation over a varying Riemann surface. We only do this in the locus of stable Higgs bundles, since this simplifies the analysis greatly. Constructing the moduli space of polystable solutions is much harder, even over a single Riemann surface, and was carried out by Fan \cite{Fan2020ConstructionOT}. Note that our result likely follows from the original paper of Simpson \cite{Simpson1992HiggsBA}, however it seems worthwhile to include an analytic proof. The proof is essentially a repeat of the original proof of Hitchin \cite{hitchin}, in a modified setup. This construction is in \S\ref{subsec:moduli-of-higgs-bundles}.
   \subsection{Jacobi operator}\label{subsec:jacobi-operator} Here we collect the definition and some properties of the Jacobi operator that we will use in the sequel, mostly without proofs.
   \begin{definition}
       If $f:M\to N$ is a harmonic map between Riemannian manifolds, the Jacobi operator $\mathcal{J}_f$, defined on $f^*TN$, is 
       \begin{align*}
           \mathcal{J}_f(V)=-\Delta V+\sum_{i=1}^n R^N(\nabla_{e_i}f, V)\nabla_{e_i}f, 
       \end{align*}
       where $(e_i:1\leq i\leq n)$ form a fibrewise orthonormal basis of $f^*TN$, $R^N$ is the curavture tensor of $N$, and $\Delta V=\mathrm{tr}\nabla^2 V$ is the Laplacian constructed from the Levi-Civita connections on $f^*TN, T^*M$ and the metric on $M$.
   \end{definition}
   The significance of the Jacobi operator comes from the following well-known fact, the proof of which we will omit.
   \begin{proposition}\label{prop:jacobi-variation-of-energy}
       Let $f:(M,g)\to (N,h)$ be a harmonic map, and let $f_t:M\times(-1, 1)\to N$ be a variation of $f=f_0$. If we let $\dot{f}=\left.\frac{\partial f_t}{\partial t}\right|_{M\times\{0\}}$, we have 
       \begin{align*}
           \frac{d^2}{dt^2} \frac{1}{2}\int_{M} \mathrm{tr}_{g} \left(f^*h\right)d\mathrm{vol}_M&=\int_M \langle \dot{f}, \mathcal{J}_f\dot{f}\rangle d\mathrm{vol}_M.
       \end{align*}
   \end{proposition}
   We will also use the formula of Micaleff--Moore \cite[equation (2.3)]{Micallef1988MinimalTA}, stated below as a proposition.
   \begin{proposition}\label{prop:micaleff-moore}
       If $M$ is a Riemann surface, equipped with a K\"ahler form $\omega$, and $f:M\to N$ is a harmonic map, then 
       \begin{align*}
           \mathcal{J}_f(V)\omega=-2i\left(\partial\bar{\partial}_AV+R(V,\partial f)\wedge\overline{\partial f}\right).
       \end{align*} 
   \end{proposition}
   Finally, we state a result of Sunada \cite[Lemma 3.4]{Sunada1979}.
   \begin{proposition}\label{prop:sunada}
       Let $(M,g)$ be a compact connected Riemannian manifold, $X$ be a non-positively curved symmetric space, and $\rho:\pi_1(M)\to \mathrm{Isom}(X)$ be a representation. Let $f:\tilde{M}\to X$ be a $\rho$-equivariant harmonic map from the universal cover $\tilde{M}$ of $M$. If $s\in\Gamma(f^*TX)$ satisfies $\mathcal{J}_fs=0$, then $f_s(x)=\mathrm{Exp}_{f(x)} s(x)$ is harmonic and $\rho$-equivariant.
   \end{proposition}
   Note that \cite[Lemma 3.4]{Sunada1979} is only stated when the image of $\rho$ acts freely and properly discontinuously on $X$, but the exact same proof shows the equivariant version in Proposition \ref{prop:sunada}.
   \begin{corollary}\label{cor:ker-coker-j}
    Let $S$ be a Riemann surface, $\rho:\pi_1(S)\to\mathrm{GL}(n,\mathbb{C})$ be an irreducible representation, and $f:S\to X_n$ be a harmonic $\rho$-equivariant map. Then $\ker\mathcal{J}_f$ is generated by the constant vector field whose value at $s\in S$ is $\mathrm{id}\cdot f(s)$, where $\mathrm{id}\in\mathrm{End}(\mathbb{C}^n)\cong\mathfrak{gl}(n,\mathbb{C})$.  
   \end{corollary}
   \begin{remark}
    Note that since $\mathcal{J}_f$ is elliptic, it does not matter which function space we are referring to, since by elliptic regularity functions in $\ker\mathcal{J}_f$ are automatically smooth. 
   \end{remark}
   \begin{proof}[Proof of Corollary \ref{cor:ker-coker-j}]
    Since $\rho$ is irreducible, the flat bundle associated to $\rho$ has a unique harmonic metric up to scaling. It follows that the harmonic map $f$ is unique up to global translation by $\lambda\mathrm{id}$ for $\lambda\in\mathbb{R}^*$. Therefore by Proposition \ref{prop:sunada}, if $s$ is in $\ker\mathcal{J}_f$, then it must be a scalar multiple of $\mathrm{id}\cdot f(s)$. Conversely, translating $f$ by $\lambda\mathrm{id}$ gives harmonic maps, so $\mathrm{id}\cdot f$ must be a Jacobi field.
   \end{proof}
   \subsection{Dictionary between the harmonic map and the harmonic metric on a flat bundle}\label{subsec:dictionary-harmonic-map-harmonic-metric}
   In this subsection, we transport the results of Proposition \ref{prop:variation-of-harmonic-eq} and of the previous subsection to equations for the first-order variation of $\mathrm{Higgs}(\rho, J)$, as we vary $J$.
   \par Let $E$ be a vector bundle over $\Sigma_g$. Pick a point $x\in\Sigma_g$ and fix one of its lifts $\tilde{x}\in\tilde{\Sigma}_g$, and a frame $e_1,e_2,...,e_n\in E_x$. Denote by $\mathrm{Hom}(\pi_1(\Sigma_g),\mathrm{GL}(n,\mathbb{C}))$ the space of homomorphisms $\pi_1(\Sigma_g)\to\mathrm{GL}(n,\mathbb{C})$. Note that this is not the same as the chararcter variety $\mathrm{Rep}(\pi_1(\Sigma_g), \mathrm{GL}(n,\mathbb{C}))$, since we consider homomorphisms differing in a conjugacy as distinct.
   \par We define the extended holonomy map 
    \begin{align*}
        \mathrm{EHol}:\left\{(D, h): \begin{matrix} D\text{ flat connection on }E \\ h\text{ Hermitian metric on }E \end{matrix}\right\}\longrightarrow \left\{(\rho, f): \begin{matrix}\rho\in\mathrm{Hom}(\pi_1(\Sigma_g), \mathrm{GL}(n,\mathbb{C})) \\ f:\tilde{\Sigma}_g\to X_n\text{ }\rho\text{-equivariant}\end{matrix}\right\}
    \end{align*}
    as follows. 
   Given a Hermitian metric $h$ and a flat connection $D$ on $E$, let $\rho$ be the holonomy of $D$ with respect to the frame $e_1,e_2,...,e_n$.  Define $f$ as follows: for $\tilde{y}\in\tilde{\Sigma}_g$ lying over $y\in\Sigma_g$, 
   \begin{align*}
       f(\tilde{y})=(h(\hat{e}_i, \hat{e}_j))_{1\leq i,j\leq n},
   \end{align*}
   where $(\hat{e}_i:1\leq i\leq n)$ is the $D$-parallel transport of the frame $(e_i:1\leq i\leq n)$ along the projection to $\Sigma_g$ of a path from $\tilde{x}$ to $\tilde{y}$. Then we set $\mathrm{EHol}(D, h)=(\rho, f)$. It is easy to see that this is a bijection. 
   \par Given a harmonic metric $h$ and a volume form $\omega$ on $\Sigma_g$, we may consider the tension field $\tau(f)\omega$ of the associated $\rho$-equivariant map $f:\tilde{\Sigma}_g\to X_n$. This is a $f^*TX_n$-valued 2-form, so when contracted with the Maurer--Cartan form $\omega^\mathfrak{p}$,  we get a 2-form taking values in the equivariant pullback of $\mathrm{GL}(n,\mathbb{C})\times_{U(n)}\mathfrak{p}$, which is precisely the bundle of self-adjoint endomorphisms $\mathfrak{p}(E)$. We denote this 2-form $\tau(h)$ and call it the tension field of the metric $h$. Note that we have introduced a volume form $\omega$ on $\Sigma_g$ to remove the dependence of the tension field on the background metric on $\Sigma_g$. We similarly contract with $\omega^\mathfrak{p}$ the operator $\mathcal{J}_f\omega$, to get a second order differential operator $\mathcal{J}_h:C^\infty(\mathfrak{p}(E))\to \Omega^2(\mathfrak{p}(E))$.
   \par Given any Hermitian metric on $E$, we can decompose any connection $D=\nabla+\Psi$, where $\nabla$ is a unitary connection on $E$, and $\Psi$ is self-adjoint $\mathrm{End}(E)$-valued 1-form with respect to $h$. Explicitly, 
   \begin{gather*}
   h(s, \Psi t)=-\frac{1}{2}(D h)(s, t),
   \nabla = D-\Psi.
   \end{gather*} 
   Here $\Psi$ will represent the derivative of the associated $\rho$-equivariant map $f$, since $\omega^\mathfrak{p}(df)=-2\Psi$.
   \begin{claim}\label{claim:tension-field-jacobian-higgs}
    When $h$ is a Hermitian metric, we have 
    \begin{align*}
        \tau(h)=4id^\nabla\left(\Psi^{1,0}\right).
    \end{align*}
    Moreover, when $h$ is harmonic, we have 
    \begin{align*}
        \mathcal{J}_h=-2i\left(\partial\bar{\partial}-[[\cdot, \Psi^{1,0}], \Psi^{0,1}]\right).
    \end{align*}
   \end{claim}
   \begin{proof}
    The first equality follows immediately from $\tau(f)\omega=-2i{\bar{\partial}}{\partial} f$ and $\omega^\mathfrak{p}(df)=-2\Psi$. The second is equivalent to Proposition \ref{prop:micaleff-moore}.
   \end{proof}
   \subsubsection{Variation of the Higgs field}
   In this subsection, we give equations for the first-order variation of a Higgs field, as we vary the underlying Riemann surface. For convenience, we assume the smoothness result Theorem \ref{thm:harmonic-mp-depends-smoothly-on-rep-cx-structure}.
   \par We first remind the reader that given a harmonic map $f:\left(\tilde{\Sigma}_g,J\right)\to X_n$, the Higgs bundle consists of the bundle $E^\mathbb{C}$ which is the equivariant pullback of $\mathrm{GL}(n,\mathbb{C})\times_{U(n)}\mathfrak{p}^\mathbb{C}$ by $f$, and of the Higgs field
   \begin{align*}
        \phi=-\frac{1}{2}\omega^\mathfrak{p}(\partial f).
   \end{align*}
   Note that $\omega^\mathfrak{p}(df)$ is by definition self-adjoint, so comparing $(0,1)$-parts of $\omega^\mathfrak{p}(df)$ and $\left(\omega^\mathfrak{p}(df)\right)^*$, we see that 
   \begin{align*}
        \phi^{*h}=-\frac{1}{2}\omega^\mathfrak{p}(\bar{\partial}f).
   \end{align*}
   The harmonic metric on $(E^\mathbb{C},\phi)$ is given by the pullback metric on $E^\mathbb{C}$, and hence the connection of the flat bundle is given by 
   \begin{align*}
    D=\nabla+\phi+\phi^{*h}.
   \end{align*}
   \begin{proposition}\label{prop:variation-higgs}
        Let $\rho:\pi_1(\Sigma_g)\to\mathrm{GL}(n,\mathbb{C})$ be an irreducible representation. Let $S_t$ be a smooth path of Riemann surfaces based at $S_0=S\in\Teich_g$, in the direction of $\mu\in\Omega^{-1,1}(S)$. Then there exists a path $(E, \bar{\partial}_t, \phi_t)$ of Higgs bundles over $S_t$, all with the same harmonic metric $h$, such that $\nabla_{\bar{\partial}_t, h}+\phi_t+\phi_t^{*h}$ has holonomy $\rho$, and 
        \begin{gather*}
            \dot{\nabla}_{\bar{\partial},h}=-[V, \phi+\phi^{*h}],\\
            \dot{\phi}=\partial V+\mu\phi-\bar{\mu}\phi^{*h},
        \end{gather*}
        where $V$ is the solution to the equation 
        \begin{align*}
            d\left(\partial V+\mu\phi-\bar{\mu}\phi^{*h}\right)=[[V, \phi^{*h}], \phi].
        \end{align*}
   \end{proposition}
   \begin{proof}
        This is just Proposition \ref{prop:variation-of-harmonic-eq} in a different guise. By Theorem \ref{thm:harmonic-mp-depends-smoothly-on-rep-cx-structure}, we get a smooth path of equivariant harmonic maps $f_t:\tilde{S}_t\to X_n$. By Proposition \ref{prop:variation-of-harmonic-eq}, we see that after identifying the pullback bundles appropriately, the metric on $f_t^*TX_n$ is constant, and the connection is varying by 
        \begin{align*}
            \dot{\nabla}&=R(\dot{f}, df).
        \end{align*}
        We also have 
        \begin{align}\label{eq:variation-partial-f}
            \frac{d}{dt}\left(\partial f\right)&=\frac{d}{dt}\left(df\circ\frac{\mathrm{id}-iJ}{2}\right)=\partial \dot{f}+\mu\partial f-\bar{\mu}\bar{\partial}f.
        \end{align}
        \par Using the Maurer--Cartan form, we get a smooth path of Higgs bundles with the same harmonic metric $(E, \bar{\partial}_t, \phi_t)$, with the Chern connection $\nabla_t$, such that 
        \begin{align*}
            \dot{\nabla}=-{[\dot{f}, \phi+\phi^{*h}]}.
        \end{align*}
        By applying the Maurer--Cartan form to (\ref{eq:variation-partial-f}), we see that 
        \begin{align*}
            \dot{\phi}&=\partial\dot{f}+\mu\phi-\bar{\mu}\phi^{*h}.
        \end{align*}
        Doing the same thing to the first equation in Proposition \ref{prop:variation-of-harmonic-eq}, we get 
        \begin{align*}
            d\dot{\phi}-[[\dot{f}, \phi^{*h}], \phi]=0.
        \end{align*}
   \end{proof}
   \subsection{Main smoothness result}\label{subsec:main-smoothness-result}
   We now show the main result of this section: the smooth dependence of the harmonic map into $X_n=\mathrm{GL}(n,\mathbb{C})/\mathrm{U}(n)$ on  the complex structure on $\Sigma_g$ and on the representation $\rho:\pi_1(\Sigma_g)\to\mathrm{GL}(n,\mathbb{C})$. We remind the reader that by $\mathrm{Hom}(\pi_1(\Sigma_g), \mathrm{GL}(n,\mathbb{C}))$ we denote the space of representations $\pi_1(\Sigma_g)\to\mathrm{GL}(n,\mathbb{C})$, and that by $\mathrm{Rep}(\pi_1(\Sigma_g), \mathrm{GL}(n,\mathbb{C}))$ we denote the corresponding character variety.
   \begin{theorem}\label{thm:harmonic-mp-depends-smoothly-on-rep-cx-structure}
        Given an irreducible representation $\rho_0:\pi_1(\Sigma_g)\to \mathrm{GL}(n,\mathbb{C})$ and an almost complex structure $J_0$ on $\Sigma_g$, there exists a smooth map 
        \begin{align*}
            f:U\times\tilde{\Sigma}_g\to X_n
        \end{align*}
        where $U$ is a neighbourhood of $(\rho_0, J_0)$ in $\mathrm{Hom}(\pi_1(\Sigma_g),\mathrm{GL}(n,\mathbb{C}))\times\Teich_g$, such that for $(\rho, J)\in U$, the map $f(\rho, J, -)$ is a $\rho$-equivariant $J$-harmonic map $\tilde{\Sigma}_g\to X_n$.
   \end{theorem} 
   \begin{remark}
    Note that by \cite[\S 1.2, Proposition]{goldman}, the space $\mathrm{Hom}(\pi_1(\Sigma_g), \mathrm{GL}(n,\mathbb{C}))$ is non-singular at representations $\rho$ whose image has centralizer equal to the center of $\mathrm{GL}(n, \mathbb{C})$, that is $\mathbb{C}\cdot\mathrm{id}$. In particular, it is smooth in a neighbourhood of any irreducible representation $\rho$.
   \end{remark}
  
   \begin{proof}[Proof of Theorem \ref{thm:harmonic-mp-depends-smoothly-on-rep-cx-structure}] Let $E$ be a complex vector bundle of rank $n$ over $\Sigma_g$. We fix a fibre of $E$ and a frame of that fibre, as in \S\ref{subsec:dictionary-harmonic-map-harmonic-metric}. Let $D_0$ be the flat connection on $E$ with holonomy $\rho_0$, as in the definition of $\mathrm{EHol}$, and let $h_0$ be the corresponding harmonic metric.  \par We abuse notation slightly and denote by $0\in\Teich_g$ the point corresponding to $J_0$. We extend $J_0$ to a family of almost complex structures $J_t$ on $\Sigma_g$, depending real analytically on $t\in\Teich_g$, such that $(\Sigma_g, J_t)$ represents the marked Riemann surface given by $t$.
    \par We let $\mathrm{Met}(E)$ be the cone of Hermitian metrics on $E$, and define 
    \begin{align*}
        \tau:\mathcal{F}\times\Teich_g\times \mathrm{Met}(E)\to \Omega^2(\mathfrak{p}(E)),
    \end{align*}
    where $\mathcal{F}=\{B\in\Omega^1(\mathrm{End}(E)):D_0B+B\wedge B=0\}$ is the space of flat connections $D_0+B$, by 
    \begin{align*}
        \tau(B, t, h)= (D_0+B-\Psi)\left(\Psi\circ\frac{\mathrm{id}-iJ_t}{2}\right),
    \end{align*}
    where 
    \begin{gather*}
        h(s, \Psi t)=-\frac{1}{2}((D_0+B)h)(s, t).
    \end{gather*}
    From Claim \ref{claim:tension-field-jacobian-higgs}, we see that $\tau(B, t, h)=0$ if and only if $\mathrm{EHol}(D_0+B, h)=(\rho, f)$, where $f$ is a $\rho$-equivariant harmonic map. It therefore suffices to construct a smooth $h=h(B, t)$ in a neighbourhood of $(0, 0)\in\mathcal{F}\times\Teich_g$ such that $\tau(B, t, h(B, t))=0$. 
    \par We extend the definition of $\tau$ to the space $\mathrm{Met}(E)_{k+\alpha+2}$ of $C^{k+\alpha+2}$ Hermitian metrics, so that its range is the space $\Omega^2(\mathfrak{p}(E))_{k+\alpha}$ of $C^{k+\alpha}$ forms. These are now Banach manifolds, and we will prove the theorem by appealing to the Banach implicit function theorem.
    \par We first observe that $\tau$ is a smooth map between Banach manifolds, and that \begin{align*}
        \frac{\partial\tau}{\partial h}=\frac{\partial^2 \mathrm{E}}{\partial h^2},
    \end{align*}
    where $\mathrm{E}$ is the energy of the harmonic map associated to $h$. As can be seen from Proposition \ref{prop:jacobi-variation-of-energy}, the derivative $\frac{\partial\tau}{\partial h}$ is the Jacobi operator $\mathcal{J}_{h_0}$ (this is analogous to \cite[Lemma 2.6]{Eells1981DeformationsOM}).
    \par We will apply the Banach manifold implicit function theorem to the following map 
    \begin{align*}
        F:\mathcal{F}\times\Teich_g\times\mathrm{Met}(E)_{k+\alpha+2}\to \Omega^2(\mathfrak{p}(E))_{k+\alpha}
    \end{align*}
    by $F(B, t, h)=\tau(B, t, h)+\log\abs{\det(h_x)}\frac{\mathrm{id}_E}{n}\omega$. Here $\omega$ is a volume form on $\Sigma_g$, and by abuse of notation, we refer to the matrix of the metric $h_x$ with respect to $(e_i:1\leq i\leq n)$ by $h_x$ as well. Note that $F$ is smooth, and that 
    \begin{align}\label{eq:partial-f-partial-h}
        \frac{\partial F}{\partial h}(\dot{h})=\mathcal{J}_{h_0}(\dot{h})+\mathrm{Re}\;\mathrm{tr}\left(h_x^{-1}\dot{h}_x\right)\frac{\mathrm{id}_E}{n}\omega.
    \end{align}Because $\mathcal{J}_{h_0}$ is self-adjoint, its image is the orthogonal complement to its kernel, and hence by Corollary \ref{cor:ker-coker-j} (and Claim \ref{claim:tension-field-jacobian-higgs}), 
    \begin{align}\label{eq:im-jf0}
        \mathrm{im}(\mathcal{J}_{h_0})&=\left\{\alpha\in\Omega^2(\mathfrak{p}(E))_{k+\alpha}: \mathrm{Re}\int_{\Sigma_g} \mathrm{tr}(\alpha)=0\right\}.
    \end{align}
    \par Therefore if $\frac{\partial F}{\partial h}(\dot{h})=0$, integrating and taking the real part of the trace of (\ref{eq:partial-f-partial-h}), we see that $\mathcal{J}_{h_0}(\dot{h})=0$ and $\mathrm{Re}\;\mathrm{tr}(h_x^{-1}\dot{h}_x)=0$. The first condition forces $h^{-1}\dot{h}$ to be a multiple of the identity, and the second to vanish. Thus $\frac{\partial F}{\partial h}$ is injective.\par We now show surjectivity of $\frac{\partial F}{\partial h}$. Let $\theta\in\Omega^2(\mathfrak{p}(E))_{k+\alpha}$ be arbitrary, and split 
    \begin{align*}
        \theta=\left(\mathrm{Re}\;\int_{\Sigma_g}\mathrm{tr}(\theta)\right)\frac{\mathrm{id}_E}{n}\omega+\xi, 
    \end{align*}
    so that $\xi$ satisfies the condition on the right-hand side of (\ref{eq:im-jf0}). Thus $\xi=\mathcal{J}_{h_0}(\dot{h})$ for some $\dot{h}$. We now replace $\dot{h}$ by $\dot{h}+\lambda h$ for some $\lambda\in\mathbb{R}$, so that 
    \begin{align*}
        \mathrm{Re}\;\mathrm{tr}\left(h_x^{-1}\dot{h}_x\right)&=\mathrm{Re}\int_{\Sigma_g}\mathrm{tr}(\theta). 
    \end{align*}
    This does not change $\mathcal{J}_{h_0}(\dot{h})$, and hence $\frac{\partial F}{\partial h}(\dot{h})=\theta$. This concludes the proof that $\frac{\partial F}{\partial h}$ is surjective.
    \par Since $\frac{\partial F}{\partial h}$ is a bijective continuous linear map, by the open mapping theorem it is an isomorphism of Banach spaces. Therefore by the Banach manifold implicit function theorem, we can construct a smooth function $h=h(B, t)$ over $U$ such that $F(B, t, h(B, t))=0$.
    \par By Stokes' theorem, as $\tau(h)=d^\nabla\Psi^{1,0}$, we have 
    \begin{align*}
        \int_{\Sigma_g} \mathrm{tr}\;\tau(B,t,h)=0.
    \end{align*}
    Therefore $F=0$ implies, after taking the trace and integrating, that $\mathrm{Re}\;\mathrm{tr}\left(h_x^{-1}\dot{h}\right)=0$ and that $\tau=0$. Therefore $h=h(\rho, t)$ is a harmonic metric. 
   \end{proof}
   \subsection{Consequences of the main smoothness result}
   \subsubsection{Proof of Theorem \ref{thm:higgs-lap-zero}}
    Let $\rho=\rho_1\oplus\rho_2\oplus...\oplus {\rho_k}$ be the decomposition of $\rho$ into its irreducible components where $\rho_i:\pi_1(\Sigma_g)\to\mathrm{GL}(V_i)$ such that $V_1\oplus V_2\oplus ...\oplus V_k=\mathbb{C}^n$. We fix a marked Riemann surface $S\in\Teich_g$. Then each $\rho_i:\pi_1(\Sigma_g)\to\mathrm{GL}(V_i)$ is irreducible, so applying Theorem \ref{thm:harmonic-mp-depends-smoothly-on-rep-cx-structure}, there exists a neighbourhood $U\subset\Teich_g$ containing $S$, along with smooth maps 
    \begin{align*}
        f_i:U\times\tilde{\Sigma}_g\to X_{n_i},
    \end{align*}
    such that $f_i(J, -)$ is a $\rho_i$-equivariant $J$-harmonic map. We combine these maps into a smooth map 
    \begin{align*}
        f=f_1\times f_2\times ...\times f_k:U\times\tilde{\Sigma}_g\to \prod_{i=1}^k\mathrm{GL}(V_i)/\mathrm{U}(V_i)\subset \mathrm{GL}(n,\mathbb{C})/\mathrm{U}(n)=X_n. 
    \end{align*}
    The map $f$ has the property that for an arbitrary $J$, the map $f(J, -)$ is a $\rho$-equivariant $J$-harmonic map $\tilde{\Sigma}_g\to X_n$. \par Now for any Beltrami form $\mu$ on $S$, we can pick a disk $\iota:\mathbb{D}\to\Teich_g$ in Teichm\"uller space with $\iota(0)=S$ and $\iota_z(0)=\mu$, and consider the complex disk of harmonic maps $(f(\iota(t), -): t\in\mathbb{D})$ based at $S$ with direction $\mu$. The characterization of $K_\rho(S)$ in Theorem \ref{thm:higgs-lap-zero} now follows immediately from Theorem \ref{thm:lap-energy-zero}. 
   \subsubsection{Smoothness of $\mathcal{R}$ on the open set of irreducible representations}\label{subsubsec:smoothness-of-r}
   Note that we will only ever use smoothness of $\mathcal{R}_\rho$ for some fixed irreducible $\rho\in\mathrm{Rep}(\pi_1(\Sigma_g), \mathrm{GL}(n,\mathbb{C}))$, but for completeness here we show smoothness of $\mathcal{R}$ over the set of irreducible representations. We first remind the reader that the character variety $\mathrm{Rep}(\pi_1(\Sigma_g), \mathrm{GL}(n,\mathbb{C}))$ is smooth near any irreducible representation (see e.g. \cite[Theorem 3]{char-varities-background}), so the question of smoothness of $\mathcal{R}$ is well-defined. 
   \par Fix an irreducible representation $\rho_0:\pi_1(\Sigma_g)\to\mathrm{GL}(n,\mathbb{C})$ and a marked Riemann surface $S_0\in\Teich_g$. Let $U\subseteq\mathrm{Hom}(\pi_1(\Sigma_g), \mathrm{GL}(n,\mathbb{C}))\times\Teich_g$ be the open set containing $(\rho_0, S_0)$ from Theorem \ref{thm:harmonic-mp-depends-smoothly-on-rep-cx-structure}. Let $f:U\times\tilde{\Sigma}_g\to X_n$ be the map from Theorem \ref{thm:harmonic-mp-depends-smoothly-on-rep-cx-structure}. Pulling back the tangent bundle of $X_n$ via $f$, we may construct Higgs bundles $(E, \bar{\partial}_{\rho, S}, \phi_{\rho, S})$ each equipped with a harmonic metric $h_{\rho, S}$, as in \S\ref{sect:prelim-nonabelian-hodge}, that vary smoothly with $(\rho, S)\in U$, such that $(E, \bar{\partial}_{\rho, S}, \phi_{\rho, S})=\mathrm{Higgs}(\rho, S)$.
   \par We let $\tilde{\mathcal{R}}(\rho, S)\in \mathrm{Hom}(\pi_1(\Sigma_g), \mathrm{GL}(n,\mathbb{C}))$ be the holonomy of $\nabla^{{\rho, S}}+i\phi_{\rho, S}-i\phi_{\rho, S}^{*h_{\rho, S}}$, where $\nabla^{{\rho, S}}$ is the Chern connection on $(E, \bar{\partial}_{\rho, S}, h_{\rho, S})$. Here we use the identifications from \S\ref{subsec:dictionary-harmonic-map-harmonic-metric} to get a representation, and not just an element of the character variety. Note that $\tilde{\mathcal{R}}(\rho, S)$ represents the Higgs field $i\cdot\mathrm{Higgs}(\rho, S)$. Since the Chern connection depends smoothly on both the Hermitian metric and holomorphic structure on $E$, we see that 
   \begin{align*}
    \tilde{\mathcal{R}}:U\longrightarrow \mathrm{Hom}(\pi_1(\Sigma_g), \mathrm{GL}(n,\mathbb{C}))
   \end{align*}
   is smooth. Let $p:\mathrm{Hom}(\pi_1(\Sigma_g), \mathrm{GL}(n,\mathbb{C}))\to \mathrm{Rep}(\pi_1(\Sigma_g), \mathrm{GL}(n,\mathbb{C}))$ be the natural quotient map. Then in the diagram 
   \[
    \begin{tikzcd}[column sep=small]
      U \arrow{r}{p\circ\tilde{\mathcal{R}}}  \arrow{d}{p\times\mathrm{id}_{\Teich_g}} 
      & \mathrm{Rep}(\pi_1(\Sigma_g), \mathrm{GL}(n,\mathbb{C})) \\
          \mathrm{Rep}(\pi_1(\Sigma_g), \mathrm{GL}(n,\mathbb{C}))\times\Teich_g \arrow{ur}{\mathcal{R}}
    \end{tikzcd}
    \]
    the vertical map $p\times\mathrm{id}_{\Teich_g}$ is a surjective submersion onto a neighbourhood of $(\rho_0, S_0)\in\mathrm{Rep}(\pi_1(\Sigma_g), \mathrm{GL}(n,\mathbb{C}))\times\Teich_g$ and the map $p\circ\tilde{\mathcal{R}}=\mathcal{R}\circ (p\times\mathrm{id}_{\Teich_g})$ is smooth. Therefore $\mathcal{R}$ is smooth in some neighbourhood of $(\rho_0, S_0)$.
   \subsection{Moduli of stable Higgs bundles}\label{subsec:moduli-of-higgs-bundles}
   Here we construct the smooth structure on the moduli space of stable Higgs bundles over a varying Riemann surface. Note that this was carried out by Fan in the general case of polystable Higgs bundles \cite{Fan2020ConstructionOT} over a single Riemann surface. We follows his analysis, though our case is significantly simpler due to the fact that stable Higgs bundles have trivial stabilizer in the relevant gauge groups.
   \subsubsection{Preliminary definitions} We fix a genus $g\geq 2$, and let $J_t$ be an almost complex structure on $\Sigma_g$ that represents the point $t\in\Teich_g$, and such that $J_t$ depends real analytically on $t$. We also fix a smooth complex vector bundle $E\to\Sigma_g$, with a smooth Hermitian metric $h$. We define the group of gauge transformations  
   \begin{gather*}
       \mathcal{G}=\{T\in\Gamma(\Sigma_g, \mathrm{GL}(E)):T\text{ is fibrewise unitary}\}.
   \end{gather*}
   We also denote by $\mathfrak{u}(E)$ the bundle of skew-adjoint endomorphisms of $E$, and let $\mathfrak{p}(E)$ be the bundle of self-adjoint endomorphisms of $E$. Let
   \begin{gather*}
       \mathcal{C}=\{(t, A, \Phi):t\in\Teich_g, \Phi\in\Omega^1(\mathfrak{p}(E))\text{ and }A\text{ is a unitary connection on }E\}, \\
       \mathcal{B}=\left\{(t, A, \Phi)\in\mathcal{C}: \begin{matrix}d_A^{0,1}\Phi^{1,0}=0\text{ with respect to the complex structure }J_t \\ d_A+\Phi\text{ is a flat connection}\end{matrix}\right\}.
   \end{gather*}
   We define $\mathcal{B}^\mathrm{s}$ as the set of triples in $\mathcal{B}$ such that $(E, d_A^{0,1}, \Phi^{1,0})$ is a stable Higgs bundle over $(\Sigma_g, J_t)$. Note that given a stable Higgs bundle $(E, \bar{\partial}^E, \phi)$ over $(\Sigma_g, J_t)$ for which $h$ is a harmonic metric, the corresponding element in $\mathcal{B}^\mathrm{s}$ is given by $(t, \nabla_{\bar{\partial}^E, h}, \phi+\phi^{*h})$, where $\nabla_{\bar{\partial}^E, h}$ is the Chern connection of $h$. The group of gauge transformations $\mathcal{G}$ admits a natural action on $\mathcal{C}$ that preserves $\mathcal{B}$ and $\mathcal{B}^\mathrm{s}$, so we set 
   \begin{align*}
       \mathcal{M}_\mathrm{Hit}^\mathrm{s}=\mathcal{B}^\mathrm{s}/\mathcal{G}.
   \end{align*}
   \subsubsection{Deformation complexes and local slices}
   We are now ready to define local slices for the action of $\mathcal{G}$ on $\mathcal{C}$ near points in $\mathcal{B}^\mathrm{s}$. The proof mirrors that of Hitchin's original paper \cite{hitchin}. 
   \par Define the deformation complex $C_\mathrm{Hit}(t,A,\Phi)$
   \begin{align*}
       \Omega^0(\mathfrak{u}(E))\stackrel{d_1}{\longrightarrow} \Omega^1(\mathfrak{u}(E))\oplus\Omega^1(\mathfrak{p}(E))\stackrel{d_2}{\longrightarrow}\Omega^2(\mathfrak{u}(E))\oplus\Omega^2(\mathrm{End}(E)),
   \end{align*}
   where 
   \begin{gather*}
       d_1\psi=(d_A\psi,[\Phi, \psi]),\\
       d_2(B, \Psi)=(d_A{B}+[\Phi,\Psi], d_A^{0,1}\Psi^{1,0}+[B^{0,1},\Phi]).
   \end{gather*}
   As shown in \cite[pp. 85, 86]{hitchin}, this is an elliptic complex that near a stable Higgs bundle has trivial zeroth and second cohomology. 
   We denote by $\mathcal{C}_k$ (resp. $\mathcal{B}^\mathrm{s}_k$) the Sobolev space $L^2_k$ of unitary connections $A$ and fields $\Phi$ satisfying the same conditions as in the definition of $\mathcal{C}$ (resp. $\mathcal{B}^\mathrm{s}$). We similarly denote by $\mathcal{G}_{k}$ the completion of $\mathcal{G}$ in the $L^2_k$ norm. Then for $k>2$, by standard theory \cite{Mitter:1979un}, 
   \begin{align*}
       \mathcal{C}_k/\mathcal{G}_{k+1}
   \end{align*}
   is a smooth Banach manifold near a point in $\mathcal{B}^\mathrm{s}$. The local slice for this action is precisely $\ker d_1^*$. 
   \par We define 
   \begin{align*}
    \mathrm{T}:\mathcal{C}&\longrightarrow \Omega^2(\mathfrak{u}(E))\oplus \Omega^2(\mathrm{End}(E))\\
    (t, A, \Phi)&\longrightarrow (d_A^2+[\Phi,\Phi], d_A^{0,1}\Phi^{1,0})
   \end{align*}
   so that $\mathcal{B}=\mathrm{T}^{-1}(0)$. Note that $\mathrm{T}$ is real analytic in $t$, and as shown by Hitchin \cite[pp. 86, 87]{hitchin}, 
   \begin{align*}
    D\mathrm{T}_{(t_0, A_0, \Phi_0)}(0, B, \Psi)=d_2(B, \Psi).
   \end{align*}
   In particular, $D\mathrm{T}_{(t_0, A_0, \Phi_0)}$ is surjective. By the Banach manifold implicit function theorem, we see that $\mathcal{M}_\mathrm{Hit}^\mathrm{s}$ is a smooth manifold. 
   \par We also provide a separate description of $T\mathcal{M}_\mathrm{Hit}^\mathrm{s}$ that will be useful in the rest of the paper. Define  
   \begin{gather*}
        \mathcal{H}^s=\left\{(t, \bar{\partial}_A, \Phi): \begin{matrix}t\in\Teich_g, \Phi\in \Omega^1(\mathfrak{p}(E))\\ \bar{\partial}_A\text{ is a }(0,1)\text{-connection on }E\text{ for }J_t\end{matrix},\text{ such that }(E, \bar{\partial}_A, \Phi^{1,0})\text{ is a stable Higgs bundle}\right\},\\
       \mathcal{M}_\mathrm{Higgs}^\mathrm{s}=\frac{\mathcal{H}^\mathrm{s}}{\Gamma(\Sigma_g, \mathrm{GL}(E))}.
   \end{gather*}
   The non-abelian Hodge correspondence provides a homeomorphism 
   \begin{align*}
       \mathcal{M}_\mathrm{Hit}^s\to \mathcal{M}_\mathrm{Higgs}^s,
   \end{align*}
   by taking $(t, A, \Phi)$ to $(t, d_A^{0,1}, \Phi)$. We use this homeomorphism to give $\mathcal{M}_\mathrm{Higgs}^s$ a smooth structure. We define another deformation complex $C_\mathrm{Higgs}(t, \bar{\partial}_A, \Phi)$
   \begin{align*}
       \Omega^0(\mathrm{End}(E))\stackrel{\bar{\partial}_A+\Phi^{1,0}}{\longrightarrow} \Omega^1(\mathrm{End}(E))\stackrel{\bar{\partial}_A+\Phi^{1,0}}{\longrightarrow} \Omega^2(\mathrm{End}(E)).
   \end{align*}
   It is easy to see that the natural isomorphism 
   \begin{align*}
       \Omega^1(\mathfrak{u}(E))\oplus\Omega^1(\mathfrak{p}(E))\to \Omega^1(\mathrm{End}(E))
   \end{align*}
   defines an isomorphism $H^1(C_\mathrm{Hit})\cong H^1(C_\mathrm{Higgs})$ (the reader can also consult \cite[\S 2]{Fan2020ConstructionOT}). This isomorphism shows that the natural map 
   \begin{align}\label{eq:moduli-higgs-hitchin-equiv-tangent}
       \frac{T_{(E, \bar{\partial}_{A_0}, \Phi_0)}\mathcal{H}^s}{(\bar{\partial}_{A_0}+\Phi_0^{1,0})\Omega^0(\mathrm{End}(E))}\to T_{(t_0, \bar{\partial}_{A_0}, \Phi_0)}\mathcal{M}_\mathrm{Higgs}^\mathrm{s}\cong T_{(t_0, {A_0}, \Phi_0)}\mathcal{M}_\mathrm{Hit}^\mathrm{s} 
   \end{align}
   is an isomorphism.
   \section{Proof of Theorem \ref{thm:main-higgs}}\label{sect:pf-main-higgs}
   Here we prove Theorem \ref{thm:main-higgs}, that is $K_\rho(S)=\ker D_S\mathcal{R}_\rho$ for some fixed marked Riemann surface $S$. We have a description of $K_\rho(S)$ from Theorem \ref{thm:higgs-lap-zero}, and hence it only remains to describe $\ker D_S\mathcal{R}_\rho$. This will follow from \S\ref{subsec:moduli-of-higgs-bundles}, in particular from the isomorphism (\ref{eq:moduli-higgs-hitchin-equiv-tangent}).
   \par We first give a preliminary description of $\ker D\mathcal{R}_\rho$ in the following proposition. 
   \begin{proposition}\label{prop:equiv-r-rho-mu=0}
    Let $V(\phi)$ be the solution to the equation from Proposition \ref{prop:variation-higgs}, and set $V=V(\phi)+iV(i\phi)$. Then $(\mathcal{R}_\rho)_*\mu=0$ if and only if there exists a section $A$ of $\mathrm{End}(E)$ such that 
    \begin{gather*}
        [\phi^{*h},V]=\bar{\partial}A, \\
        \partial V-2\bar{\mu}\phi^{*h}=[\phi, A].
    \end{gather*} 
    \end{proposition}
   \begin{proof}
    We consider two first order variations of $(d_A, i\phi)$ in $\mathcal{M}_\mathrm{Hit}^\mathrm{s}$, given by 
    \begin{align*}
        -[V(\phi), \phi^{*h}+\phi]&, i\partial V(\phi)+i\mu\phi-i\bar{\mu}\phi^{*h}, \\
        -[V(i\phi), i\phi-i\phi^{*h}]&, \partial V(i\phi)+i\mu\phi+i\bar{\mu}\phi^{*h}.
    \end{align*}
    By Proposition \ref{prop:variation-higgs}, the first is tangent to the path $i\cdot\mathrm{Higgs}(\rho, J_t)$, and the second one is tangent to $\mathrm{Higgs}(\mathcal{R}_\rho(J_0), J_t)$. Subtracting, we get a first order variation
    \begin{align}\label{eq:variation-diff}
        -[V, \phi^{*h}]-[V^{*h}, \phi], i\left(\partial V-2\bar{\mu}\phi^{*h}\right).
    \end{align}
    \begin{claim}\label{claim:baby-equiv-rrho}
        We have $(\mathcal{R}_\rho)_*\mu=0$ if and only if the variation (\ref{eq:variation-diff}) vanishes as a tangent vector to $\mathcal{M}_\mathrm{Hit}^\mathrm{s}$.
    \end{claim}
    \begin{proof}
    \par Suppose first that $(\mathcal{R}_\rho)_*\mu=0$. Let $J_t$ be a smooth path of almost complex structures on $\Sigma_g$, where we identify $(\Sigma_g, J_0)$ with $S$, tangent to $\mu\in T_S\Teich_g$. Let $\rho_t$ be a path in $\mathrm{Hom}(\pi_1(\Sigma_g), \mathrm{GL}(n, \mathbb{C}))$ that is a lift of $\mathcal{R}_\rho(J_t)$ such that $\left(\frac{d}{dt}\right)_{t=0}\rho_t=0$. 
    By Theorem \ref{thm:harmonic-mp-depends-smoothly-on-rep-cx-structure}, there exists a smooth family of maps $f_{t,s}:\tilde{\Sigma}_g\to X_n$ such that $f_{t,s}$ is a $\rho_t$-equivariant $J_s$-harmonic map. Moreover, since $\left(\frac{d}{dt}\right)_{t=0}\rho_t=0$, we get \[\left(\frac{\partial}{\partial t}\right)_{t=s=0}f=0.\] 
    It follows that the paths $f(t, t)$ and $f(0,t)$ agree to first order. Thus $i\cdot\mathrm{Higgs}(\rho, J_t)=\mathrm{Higgs}(\mathcal{R}_\rho(J_t), J_t)$ and $\mathrm{Higgs}(\mathcal{R}_\rho(J_0), J_t)$ agree to first order (in $\mathcal{M}_\mathrm{Hit}^\mathrm{s}$) at $t=0$. In particular, the variation (\ref{eq:variation-diff}) vanishes. 
    \par Conversely, assume that (\ref{eq:variation-diff}) vanishes as an element of $T\mathcal{M}_\mathrm{Hit}^\mathrm{s}$. Then the paths $\mathrm{Higgs}(\mathcal{R}_\rho(J_t), J_t)=i\cdot\mathrm{Higgs}(\rho, J_t)$ and $\mathrm{Higgs}(\mathcal{R}_\rho(J_0), J_t)$ agree to first order at $t=0$. But the representations $\mathcal{R}_\rho(J_t), \mathcal{R}_\rho(J_0)$ can be recovered as holonomies of $d_{A_t}+\phi_t+\phi_t^{*h}$, and are hence smooth functions on $\mathcal{M}_\mathrm{Hit}^\mathrm{s}$. Thus $\mathcal{R}_\rho(J_t)$ and $\mathcal{R}_\rho(J_0)$ agree to first order at $t=0$. This is exactly equivalent to $(\mathcal{R}_\rho)_*\mu=0$. 
    \end{proof}
    By the isomorphism (\ref{eq:moduli-higgs-hitchin-equiv-tangent}) from \S\ref{subsec:moduli-of-higgs-bundles}, the variation (\ref{eq:variation-diff}) vanishes if and only if 
    \begin{gather*}
        -[V,\phi^{*h}]=\bar{\partial}A,\\
        i\left(\partial V-2\bar{\mu}\phi^{*h}\right)=i[\phi,A],
    \end{gather*} 
    for some section $A$ of $\mathrm{End}(E)$. This concludes the proof of the proposition. 
   \end{proof} 
   \subsection{Only if}\label{subsec:main-higgs-only-if} Suppose that $\mathrm{E}_\rho$ has vanishing Laplacian in the direction $\mu$. Let $(E,\phi)$ be the corresponding Higgs bundle over $S$, and $h$ be the harmonic metric associated to $\rho$. By Theorem \ref{thm:lap-energy-zero}, we see that there exists a section $\xi$ of $\mathrm{End}(E)$, such that 
    \begin{align}\label{eq:pf-thm-higgs-equality}
        \mu\phi=\bar{\partial}\xi\text{ and }[\xi,\phi]=0.
    \end{align}
    From Proposition \ref{prop:variation-higgs}
    \begin{gather*}
        d\left(\partial V(\phi)+\mu\phi-\bar{\mu}\phi^{*h}\right)=[[V(\phi), \phi^{*h}], \phi],\\
        d\left(\partial V(i\phi)+i\mu\phi+i\bar{\mu}\phi^{*h}\right)=[[V(i\phi), \phi^{*h}], \phi].
    \end{gather*}
    Combining these two equations, and setting $V=V(\phi)+iV(i\phi)$, we get
    \begin{align*}
        \bar{\partial}\left({\partial}V-2\bar{\mu}\phi^{*h}\right)=[[V, \phi^{*h}],\phi].
    \end{align*}
    Taking the adjoint with respect to $h$, we get 
    \begin{align*}
        \partial\bar{\partial}\left(V^{*h}-2\xi\right)=-[[\phi, V^{*h}], \phi^{*h}].
    \end{align*}
    Since $[\xi,\phi]=0$, we see that $\partial\bar{\partial}(V^{*h}-2\xi)-[[V^{*h}-2\xi, \phi]\wedge\phi^{*h}]=0$. Note that this is the exact equation we got for $V$ in the proof of Theorem \ref{thm:lap-energy-zero}. From the Bochner argument in the last paragraph of the proof of Theorem \ref{thm:lap-energy-zero} and Claim \ref{claim:curv-la}, we get 
    \begin{align*}
        \bar{\partial}(V^{*h}-2\xi)=0\text{ and }[V^{*h},\phi]=0.
    \end{align*}
    But $[V^{*h},\phi]=0$, so $[V, \phi^{*h}]=0$, and $\bar{\partial}V^{*h}=2\bar{\partial}\xi=2{\mu}\phi$, so $\partial V^{*h}=2\bar{\mu} \phi^{*h}$. Thus $(\mathcal{R}_\rho)_*\mu=0$ by Proposition \ref{prop:equiv-r-rho-mu=0}.
    
   \subsection{If}\label{subsec:main-higgs-if}
   Let $(E, \phi)$ be the stable Higgs bundle that corresponds to $\rho$, with the harmonic metric $h$ and Chern connection $d_A$.
    By Proposition \ref{prop:equiv-r-rho-mu=0}, there exists a section $A\in\Gamma(\mathrm{End}(E))$, with the property that
   \begin{align*}
       \left(-[V, \phi^{*h}], i\left(\partial V-2\bar{\mu}\phi^{*h}\right)\right)=\left(\bar{\partial}A,  i[\phi, A]\right).
   \end{align*}
   Therefore
   \begin{align*}
       \partial\bar{\partial}A&=-[\partial V, \phi^{*h}]=-[[\phi, A]+2\bar{\mu}\phi^{*h}, \phi^{*h}]\\
       &=-[[\phi, A], \phi^{*h}].
   \end{align*}
   Thus $\partial\bar{\partial}A-[[A, \phi], \phi^{*h}]=0$, so by a Bochner argument and Claim \ref{claim:curv-la}, we get $\bar{\partial}A=0$ and $[A,\phi]=0$. Therefore 
   \begin{align*}
       [\phi, V^{*h}]=0\text{ and }\bar{\partial}V^{*h}=2\mu\phi.
   \end{align*} 
   By Theorem \ref{thm:lap-energy-zero}, the Laplacian of $\mathrm{E}_\rho$ in the direction $\mu$ vanishes, so we are done.
   \section{Relationship to the critical points of the Hitchin fibration}\label{sect:hitchin-fib-proof}
   In this section, we show Theorem \ref{thm:hitchin-integrable-system}. This follows easily from Proposition \ref{prop:hamiltonian-field-vanishes}, that we show below.
   \par We first introduce some notation. Given a Riemann surface $S$ and a Beltrami form $\mu$ on $S$, we define the function 
   \begin{align*}
    F_\mu:\mathcal{M}_\mathrm{Higgs}^\mathrm{ps}(S)&\longrightarrow \mathbb{C}\\
    (E, \phi)&\longrightarrow \int_S \mathrm{tr}(\phi\wedge \mu\phi).
   \end{align*}
   Note that $F_\mu$ factors through the Hitchin fibration $H$.
   \begin{proposition}\label{prop:hamiltonian-field-vanishes}
    Let $S$ be a Riemann surface of genus $g$, and $\mu$ be a Beltrami form on $S$. If there exists a section $\xi$ of $\mathrm{End}(E)$ such that 
    \begin{align*}
        \mu\phi=\bar{\partial}\xi\text{ and }[\phi,\xi]=0,
    \end{align*}
    then $(E, \phi)$ is a critical point for $F_\mu$.
   \end{proposition}
   \begin{proof}
    We pick a smooth path $(E, \bar{\partial}_{E, t}, \phi_t)=(E, \bar{\partial}_E+t\dot{A}+O(t^2), \phi+t\dot{\phi}+O(t^2))$ of Higgs bundles starting from $(E, \bar{\partial}_E, \phi)$. Then taking derivatives of $\bar{\partial}_{E,t}\phi_t=0$, we see that 
    \begin{gather}\label{eq:phi-hol-dot}
        \bar{\partial}_E\dot{\phi}+[\dot{A}, \phi]=0.
    \end{gather}
    Taking derivatives of $F_\mu(E, \bar{\partial}_{E,t}, \phi_t)$, we get
    \begin{gather*}
        \dot{F}_\mu=2\int_S \mathrm{tr}(\dot{\phi}\wedge\mu\phi).
    \end{gather*}
    Assume $\mu\phi=\bar{\partial}\xi$ and $[\phi,\xi]=0$ for some $\xi\in\Gamma(\mathrm{End}(E))$. Then we have 
        \begin{align*}
            \dot{F}_\mu&=2\int_S \mathrm{tr}(\dot{\phi}\wedge\bar{\partial}\xi)=2\int_S \mathrm{tr}((\bar{\partial}\dot{\phi})\xi-\bar{\partial}(\dot{\phi}\xi))\\
            &=-2\int_S \mathrm{tr}([\dot{A}, \phi]\xi)=-2\int_S \mathrm{tr}(\dot{A}[\phi, \xi])=0.
        \end{align*}
        Here we used Stokes' theorem and (\ref{eq:phi-hol-dot}) in going from the first to the second line. 
   \end{proof}
   \subsection{Proof of Theorem \ref{thm:hitchin-integrable-system}}
   Suppose now that $d=\dim K_\rho(S)$ and $(E,\phi)=\mathrm{Higgs}(\rho, S)$. Let $\mu_1,...,\mu_d$ be linearly independent Beltrami forms in $K_\rho(S)$. Then from Proposition \ref{prop:hamiltonian-field-vanishes}, we see that $F_{\mu_i}$ all have critical points at $(E,\phi)$. Note that 
   \begin{align*}
    (F_{\mu_1}, F_{\mu_2},...,F_{\mu_d}) = F\circ H,
   \end{align*}
   where $F:\bigoplus_{i=1}^n H^0(S, K_S^i)\to \mathbb{C}^d$ is given by \[ F(\phi_1, \phi_2, ..., \phi_n)=\left(\int_S \mu_i \phi_2: i=1,2,...,d\right). \]
   Note that $\nabla_{(E,\phi)}(F\circ H)=0$, since all $F_{\mu_i}$ have critical points at $(E,\phi)$. Since the $\{\mu_i:i=1,2,...,d\}$ are linearly independent, $F$ is a linear map of full rank, and thus the rank of $\nabla_{(E,\phi)}H$ is at most $\dim\mathcal{B}(S)-d$. 
   \section{The case $n=1$}\label{sect:n=1}
   Here we show Proposition \ref{prop:n=1}. In \S\ref{subsec:higgs-field-rank-1}, we show that $\phi$ from the statement of Proposition \ref{prop:n=1} corresponds to the Higgs field associated to $\rho$. In \S\ref{subsec:proof-prop-n=1}, we derive Proposition \ref{prop:n=1}. Characterization of $K_\rho$ follows from Theorem \ref{thm:higgs-lap-zero}, while integrability of the distribution $K_\rho$ follows from Theorem \ref{thm:main-higgs}.
    \subsection{The Higgs field in rank 1}\label{subsec:higgs-field-rank-1}
    We remark that a rank 1 Higgs bundle over a Riemann surface $S$ is simply a pair $(L,\phi)$ consisting of a line bundle $L$, and a holomorphic 1-form $\phi$ on $S$. This follows from the fact that $\mathrm{End}(L)\cong L\otimes L^{-1}$ is canonically isomorphic to the trivial line bundle $\mathcal{O}_S$. 
    \par Given a rank 1 Higgs bundle $(L, \phi)$, the Hitchin equation is 
    \begin{align*}
        F_\nabla=0,
    \end{align*}
    where $\nabla$ is the Chern connection of a Hermitian metric on $L$. Since $\nabla$ is unitary, its holonomy $\rho_\nabla:\pi_1(\Sigma_g)\to\mathbb{C}^*$ has image in the unitary group $\mathrm{U}(1)=S^1\leq\mathbb{C}^*$.  For $\rho:\pi_1(\Sigma_g)\to\mathbb{C}^*$, the flat connection on $(L,\phi)=\mathrm{Higgs}(\rho, S)$ with holonomy $\rho$ is $\nabla+\phi+\phi^*=\nabla+\phi+\bar{\phi}$, and hence
    \begin{align*}
        \rho(\gamma)=\rho_\nabla(\gamma)e^{-\int_\gamma \phi+\bar{\phi}}=\rho_\nabla(\gamma)e^{-2\int_\gamma\mathrm{Re}(\phi)}.
    \end{align*}
    Therefore $\int_\gamma\mathrm{Re}(\phi)=-\frac{1}{2}\log\abs{\rho(\gamma)}$.
    \subsection{Proof of Proposition \ref{prop:n=1}}\label{subsec:proof-prop-n=1} Given a rank 1 Higgs bundle $(L,\phi)=\mathrm{Higgs}(\rho, S)$, by Theorem \ref{thm:higgs-lap-zero}, the kernel of the Levi form of $\mathrm{E}_\rho$ is precisely the space 
    \begin{align*}
        \{\mu\in T_S\Teich_g:\mu\phi\text{ is }\bar{\partial}\text{-exact}\}.
    \end{align*}
    By Serre duality, this is exactly equal to 
    \begin{align*}
        \left\{\mu\in T_S\Teich_g:\int_S \mu\phi\theta=0\text{ for all }\theta\in H^0(K_S)\right\}=\left(\phi \otimes H^0(K_S)\right)^\perp\leq H^0(K_S^2)^\vee.
    \end{align*}
    Here $H^0(K_S^2)=\mathrm{QD}(S)$ is the space of holomorphic quadratic differentials on $S$, and $H^0(K_S^2)^\vee$ denotes its dual. Since $T_S\Teich_g$ is precisely the dual of $H^0(K_S^2)$, we see that the nullity of the Levi form of $\mathrm{E}_\rho$ is equal to $3g-3-\dim(\phi\otimes H^0(K_S))=2g-3$.
    \par We now show that the distribution $K_\rho$ on $\Teich_g$ is integrable. This is a complex distribution of constant dimension. Moreover by Theorem \ref{thm:main-higgs}, $K_\rho$ consists of all vectors that are annihilated by $\mathcal{R}_\rho$. Therefore if we write locally in some coordinate system on $\mathrm{Rep}(\pi_1(\Sigma_g),\mathbb{C}^*)$, \[\mathcal{R}_\rho=(f_1,f_2,...,f_N)\]
    we see that $K_\rho$ is the vanishing locus of $(df_1,df_2,...,df_N)$. By the Frobenius integrability theorem, since $\dim K_\rho$ is constant, the distribution $K_\rho$ is integrable. Since $K_\rho$ is a complex distribution, the leaves of the resulting foliation are complex submanifolds. 
   \section{The case $n\geq 2$} \label{sect:generic-particular}
   In this section, we show Theorem \ref{thm:generic-special}.
   \par In \S\ref{subsec:reps-smooth-fibre-hitchin}, we show that representations in the smooth fibre of the Hitchin fibration over $S$ have $K_\rho(S)=\{0\}$. As described in the outline, from the BNR correspondence \cite{Beauville1989SpectralCA} and Theorem \ref{thm:higgs-lap-zero}, it follows that $K_\rho(S)$ only depends on which Hitchin fibre $\mathrm{Higgs}(\rho, S)$ belongs to. We then construct strictly plurisubharmonic representations in each fibre. This is done by slightly modifying the $\mathrm{SL}(n,\mathbb{R})$ Hitchin section. Note that in the $\mathrm{SL}(n,\mathbb{R})$ Hitchin section, the energy $\mathrm{E}_\rho$ is strictly plurisubharmonic by \cite{Slegers-strict}, and our result is a slight extension of \cite{Slegers-strict}. 
   \par In \S\ref{subsec:examples}, we construct explicitly stable $\mathrm{SL}(n,\mathbb{C})$-Higgs bundles $(E,\phi)$ in the nilpotent cone over an arbitrary $S\in\Teich_g$, for any $g\geq 4, n\geq 2$, such that 
   \begin{align*}
    \dim\{\mu\in T_S\Teich_g: \text{(\ref{eq:higgs-lap-zero}) has a solution}\}\geq g-3.
   \end{align*}
   By applying the inverse of the non-abelian Hodge correspondence and Theorem \ref{thm:higgs-lap-zero}, this gives representations $\rho:\pi_1(\Sigma_g)\to\mathrm{GL}(n,\mathbb{C})$ such that $\dim K_\rho(S)\geq g-3$.
   \subsection{Representations in the smooth fibre of the Hitchin fibration}\label{subsec:reps-smooth-fibre-hitchin} Let $(E,\phi)=\mathrm{Higgs}(\rho, S)$, and $S_a$ be the spectral curve over $S$ associated to $a=H(E,\phi)$. Note that $S_a$ only depends on the fibre $a\in\bigoplus_{i=1}^nH^0(K_S^i)$, and that it is smooth for a generic $a$, as explained in \S\ref{sect:prelim}. 
    \par From the BNR correspondence \cite[Proposition 3.6]{Beauville1989SpectralCA}, the bundle $E$ can be written as the pushforward along $p$ of a line bundle $L$ on $S_a$. Moreover $S_a$ admits a holomorphic 1-form $\Phi_a$, such that $\phi$ is precisely the pushforward of $\Phi_a$ along $p$. Finally, $\Phi_a$ also only depends on $a$. 
    \par At a point $x\in S_a$ that is not a preimage of a branch point, the characteristic polynomial at $p(x)$ does not have repeated roots. Thus, if $[\xi,\phi]=0$, then $\xi$ is the pushforward of a function on $S_a$, in a neighbourhood of any point not lying over a branch point. By continuity of $\xi$, there is a global function $g$ on $S_a$ such that $\xi=p_*g$. 
    \par From Theorem \ref{thm:higgs-lap-zero}, it now follows that $\mathrm{E}_\rho$ is not strictly plurisubharmonic at $S$ in a direction $\mu$ if and only if $\mu\Phi_a$ is $\bar{\partial}$-exact. However this condition only depends on $a$, so it suffices to construct a Higgs bundle $\mathrm{Higgs}(\rho,S)=(E,\phi)\in H^{-1}(a)$ such that $\mathrm{E}_\rho$ is strictly plurisubharmonic at $S$. The rest of this subsection is devoted to this construction.
    \par Let $s:\bigoplus_{i=2}^n H^0(S, K_S^i)\to\mathcal{M}_\mathrm{Higgs}(S)$ be the Hitchin section. We extend this section to a function $\hat{s}:\bigoplus_{i=1}^n H^0(S, K_S^i)\to\mathcal{M}_\mathrm{Higgs}(S)$ by 
    \begin{align*}
        \hat{s}(q_1,q_2,...,q_n)=q_1\cdot\mathrm{id}_E+\phi,
    \end{align*} 
    where $(E,\phi)=s(q_2,q_3...,q_n)$. Note that the characteristic polynomial of $\hat{s}$ is 
    \begin{align*}
        \chi_{\hat{s}}(x)=\det(x\mathrm{id}_E-\hat{s})=\det((x-q_1)\mathrm{id}_E-s)=\chi_s(x-q_1),
    \end{align*} 
    Since the coefficients of $\chi_s$ run over all possible elements of $\bigoplus_{i=2}^n H^0(S, K_S^i)$, the coefficients of $\chi_{\hat{s}}$ run over all elements of $\bigoplus_{i=1}^n H^0(S, K_S^i)$. The proof is then concluded once the following claim is shown.
    \begin{claim}
        Let $(E,\phi)$ be in the image of $\hat{s}$. If $\mu\in \Omega^{-1,1}(S)$ is a Beltrami form such that $\mu\phi=\bar{\partial}\xi$ for some $\xi\in\mathrm{End}(E)$, then $\mu$ represents the zero direction in $T_S\Teich_g$.
    \end{claim}
    \begin{proof}
        Taking the traceless part of the equation $\mu\phi=\bar{\partial}\xi$, we see that 
        \begin{align*}
            \mu\left(\phi-\mathrm{tr}(\phi)\frac{\mathrm{id}_E}{n}\right)=\bar{\partial}\left(\xi-\mathrm{tr}(\xi)\frac{\mathrm{id}_E}{n}\right).
        \end{align*}
        But $\phi-\mathrm{tr}(\phi)\frac{\mathrm{id}_E}{n}$ is in the Hitchin section by construction. In this setting, Slegers \cite[Proof of Theorem 1.1, pp. 9]{Slegers-strict} has shown that $\mu$ represents the zero direction in $T_S\Teich_g$.
    \end{proof}
   \subsection{Examples of non-strictly plurisubharmonic representations}\label{subsec:examples}
   We assume in this section that $g\geq 4$, and fix an arbitrary Riemann surface $S\in\Teich_g$. We will construct for all $n\geq 2$ a representation $\rho:\pi_1(\Sigma_g)\to\mathrm{SL}(n,\mathbb{C})$ in the nilpotent cone, such that $\mathrm{E}_\rho$ is not strictly plurisubharmonic at $S$.
   \par We first show the construction when $n$ is odd, and then explain how to modify it in the case of $n$ even.
   \subsubsection{The odd case} We assume first that $n$ is odd. Let $p\in S$ be a generic point, and set $L=\mathcal{O}(p)$. By the geometric Riemann--Roch theorem \cite[pp. 12]{arbarello2013geometry}, $h^0(KL^{-1})=i(p)=g-2$. Here we denote by $h^i$ the dimension of the $i$-th cohomology group $H^i$, and $i(D)$ is the index of specialty of the effective divisor $D$. Therefore we fix a non-zero element $\psi\in H^0(KL^{-1})$. 
   \par We consider the following sequence 
   \begin{align}\label{eq:line-bundles}
    L_1=L^{\lfloor \frac{n-1}{2}\rfloor}, L_2=L^{\lfloor \frac{n-1}{2}\rfloor - 1},..., L_n=L^{-\lfloor \frac{n-1}{2}\rfloor},
   \end{align}
   so that there are in total $n$ line bundles. Let $E=L_1\oplus L_2\oplus...\oplus L_n$, and $\phi\in H^0(K\mathrm{End}(E))$ be the nilpotent Higgs field that has 
   \begin{align*}
    \phi:L_i\to L_{i+1}K\text{ is multiplication by }\psi,
   \end{align*} 
   and $\phi(L_n)=0$.
   \begin{claim}\label{claim:stable}
    The pair $(E,\phi)$ is a stable $\mathrm{SL}(n,\mathbb{C})$-Higgs bundle. 
   \end{claim}
   \begin{proof}
    We first observe that $\bigwedge^n E=L_1L_2...L_n\cong \mathcal{O}_S$, and that $\phi$ is subdiagonal in the basis defined by $L_1, L_2,...,L_n$, and thus traceless. Hence $(E,\phi)$ is an $\mathrm{SL}(n,\mathbb{C})$-Higgs bundle. 
    \par To show stability, suppose that $F\leq E$ is a non-zero $\phi$-invariant subbundle. Since $\phi$ is nilpotent, there exists a maximal $k\geq 0$ such that $\phi^k F\neq 0$. Then $\phi^{k}F\leq (\ker \phi)K^k=L_nK^k$. Thus $\phi^kF\leq K^kF$ has non-zero intersection with $L_nK^k$, and hence $L_n\leq F$. Repeating the argument with $E/L_n$ in place of $E$, by induction it follows that for some $i$, we have 
    \begin{align*}
        F=\bigoplus_{j\geq i}L_j.
    \end{align*} 
    But unless $F=E$, the degree of $F$ is negative, and hence so is the slope. Since the degree of $E$ is zero, $E$ is stable. 
   \end{proof}
   \par Let $\mu\in\Omega^{-1,1}(S)$ be an arbitrary Beltrami form, and assume that $\mu\psi=\bar{\partial}\xi$, for some $\xi\in\Gamma(L^{-1})$. Then we can define $\theta\in \Gamma(\mathrm{End}(E))$ by letting $\theta:L_i\to L_{i+1}$ act by multiplication by $\xi$. Such a $\theta$ commutes with $\phi$, and has $\bar{\partial}\theta=\mu\phi$. 
   \par Therefore if $\mu\psi$ is $\bar{\partial}$-exact, $\mu\in K_\rho(S)$. By Serre duality, $\mu\psi$ is $\bar{\partial}$-exact if and only if 
   \begin{align}\label{eq:belt-condition-ex-construction}
    \int_S \mu \psi\wedge\eta=0\text{ for all }\eta\in H^0(KL).
   \end{align}
   Since $T_S\Teich_g$ is dual to the space of holomorphic quadratic differentials on $S$, the codimension of the space of equivalence classes of Beltrami forms $[\mu]\in T_S\Teich_g$ for which (\ref{eq:belt-condition-ex-construction}) holds is 
   \begin{align*}
    \dim\{\psi\eta:\eta\in H^0(KL)\}=h^0(KL).
   \end{align*}
   Note that $h^0(KL)$ is the space of meromorphic 1-forms on $S$ with at most a single pole at $p\in S$. But by Stokes' theorem, the residues of the poles of a meromorphic 1-form must sum to zero. Thus forms with a single pole are in fact holomorphic, so $h^0(KL)=h^0(K)=g$. Thus the space of Beltrami forms for which (\ref{eq:belt-condition-ex-construction}) holds is precisely $3g-3-h^0(KL)=2g-3$. This concludes the case of odd $n$.
   \subsubsection{The even case} We now explain how to modify the above construction in the case of even $n$. We again pick a generic point $p\in S$, set $L=\mathcal{O}(p)$, and let $\psi$ be a non-zero section of $H^0(KL^{-1})$. Set $n=2m$, and 
   \begin{gather*}
    L_1=L^{m-1}, L_2=L^{m-2}, ..., L_m=\mathcal{O}_S,\\
     L_{m+1}=\mathcal{O}_S, L_{m+2}=L^{-1},...,L_{2m}=L^{-(m-1)}.
   \end{gather*}
   Let $\omega$ be an arbitrary holomorphic 1-form. Let $E=\bigoplus_{i=1}^{2m} L_i$, and define $\phi\in H^0(K\mathrm{End}(E))$ by 
   \begin{gather*}
    \phi:L_i\to L_{i+1}K\text{ is multiplication by }\psi\text{ unless }i=m, \\
    \phi:L_m\cong\mathcal{O}_S\to K\cong L_{m+1}K\text{ is multiplication by }\omega.
   \end{gather*}  
   The proof of Claim \ref{claim:stable} applies verbatim to show that $(E, \phi)$ is a stable $\mathrm{SL}(n,\mathbb{C})$-Higgs bundle. 
   \par An analogous argument to the one that led to (\ref{eq:belt-condition-ex-construction}), shows that $\Delta_\mu\mathrm{E}_\rho=0$ if both $\mu\psi$ and $\mu\omega$ are $\bar{\partial}$-exact. Again by Serre duality, this is equivalent to 
   \begin{align*}
    \int_S \mu\Phi=0\text{ for all }\Phi\in \psi H^0(KL)+\omega H^0(K).
   \end{align*}
   The same argument shows that $h^0(KL)=h^0(K)=g$, and hence \[ \dim K_\rho(S)\geq 3g-3-h^0(KL)-h^0(K)=g-3.\]
       
\bibliographystyle{amsplain}
    \bibliography{main}
\end{document}